\documentclass[12pt, reqno]{amsart}
\usepackage{amssymb, amsthm, amsmath, amsfonts, bbm}
\usepackage[shortlabels]{enumitem}
\usepackage{array}	
\setlist*{label=\alph*), ref=\alph*, wide, labelwidth=!}

\usepackage{color}

\usepackage[hidelinks]{hyperref}
\usepackage{esvect}

\usepackage{anysize}
\usepackage[pdftex]{graphicx}

\usepackage[makeroom]{cancel} %crossing out a formula

  \evensidemargin
\oddsidemargin
\newcommand{\argmax}{\mathop{\mathrm{argmax}}\nolimits}

\newcommand{\Var}{\mathop{\mathrm{Var}}\nolimits}

\newcommand{\supp}{\mathop{\mathrm{supp}}\nolimits}

\newcommand{\seq}{\mathop{\mathrm{seq}}\nolimits}

\makeatletter
\DeclareFontFamily{U}{tipa}{}
\DeclareFontShape{U}{tipa}{m}{n}{<->tipa10}{}
\newcommand{\arc@char}{{\usefont{U}{tipa}{m}{n}\symbol{62}}}%

\newcommand{\arc}[1]{\mathpalette\arc@arc{#1}}

\newcommand{\arc@arc}[2]{%
  \sbox0{$\m@th#1#2$}%
  \vbox{
    \hbox{\resizebox{\wd0}{\height}{\arc@char}}
    \nointerlineskip
    \box0
  }%
}
\makeatother

\DeclareMathOperator{\conv}{conv}

\DeclareMathOperator{\cl}{cl}

\DeclareMathOperator{\intr}{int}

\newtheorem{theorem}{Theorem}[section]
\newtheorem{theorema}{Theorem}

\newtheorem{prop}[theorem]{Proposition}     
\newtheorem{lemma}[theorem]{Lemma}

\newtheorem{cor}[theorem]{Corollary}

\theoremstyle{remark}
\newtheorem{remark}[theorem]{Remark}

\def\R{\mathbb{R}}
\def\Z{\mathbb{Z}}
\def\N{\mathbb{N}}
\def\E{\mathbb{E}}
\def\P{\mathbb{P}}

\def\S{\mathbb{S}}
\def\I{\mathbbm{1}}

\begin{document}

\title[Contraction principle for trajectories of random walks]{Contraction principle for trajectories of random \\ walks and Cram\'er's theorem for kernel-weighted sums}

\author{Vladislav~Vysotsky}
\address{Vladislav~Vysotsky, University of Sussex and St.\ Petersburg Department of Steklov Mathematical Institute}
\email{v.vysotskiy@sussex.ac.uk}
\thanks{This work was supported in part by the RFBR Grant 19-01-00356.}

\begin{abstract}
In 2013 A.A.\ Borovkov and A.A.\ Mogulskii proved a weaker-than-standard ``metric'' large deviations principle (LDP) for trajectories of random walks in $\R^d$ whose increments have the Laplace transform finite in a neighbourhood of zero. We prove that general metric LDPs are preserved under uniformly continuous mappings. This allows us to transform the result of Borovkov and Mogulskii into standard LDPs. We also give an explicit integral representation of the rate function they found. As an application, we extend the classical Cram\'er theorem by proving an LPD for kernel-weighted sums of i.i.d.\ random vectors in $\R^d$. 
%Moreover, we specify certain types of sets on which the weaker bound in %the a metric LDP equals the standard one. 
\end{abstract}


\maketitle

\section{Introduction}
The study of large deviations of trajectories of random walks was initiated by A.A.\ Borovkov in the 1960's. In 1976 A.A.\ Mogulskii~\cite{Mogulskii} proved a large deviations result for the trajectories of a multidimensional random walk under the assumption that the Laplace transform of its increments is finite. In~\cite{Mogulskii} Mogulskii also studied the large deviations under the weaker {\it  Cram\'er moment assumption}, i.e.\ when the Laplace transform is finite in a neighbourhood of zero, but these results appear to be of a significantly limited use. % due to their implicitness.

The further progress was due to the concept of {\it metric} large deviations principles  (LDPs, in short) on general metric spaces introduced by Borovkov and Mogulskii in~\cite{BorovkovMogulskii2010} in 2010. The upper bound in such metric LDPs is worse than the conventional one  -- the infimum of the rate function is taken over shrinking $\varepsilon$-neighbourhoods of a set rather than over its closure as in standard LDPs; compare definitions~\eqref{eq: LD general} and~\eqref{eq: LD extended} below. These bounds may differ when the rate function  is  {\it not tight}, i.e.\ its sub-level sets are non-compact.

For the scaled trajectories of random walks under the Cram\'er moment assumption, Borovkov and Mogulskii~\cite{BorovkovMogulskii2, BorovkovMogulskii3} obtained a metric LDP  in the space $D[0,1]$ of c\`adl\`ag functions equipped with a Skorokhod-type metric generating a 
%mildly modified 
version of the topology $M_1$; see Theorem~\ref{thma: BM}. Their rate function has totally bounded sub-level sets but in general, it is not tight because the metric space is not complete. There is no discussion in~\cite{BorovkovMogulskii2,BorovkovMogulskii3} on whether the usual LDP for the trajectories fails.

 Our paper was motivated by the question whether the metric LDP of~\cite{BorovkovMogulskii2,BorovkovMogulskii3} can be converted, in any reasonable sense, to a standard LDP.
%, in the original topologies if possible. 
Our progress in this direction is as follows.

Our first result is a {\it contraction principle} for metric LDPs on general metric spaces (Theorem~\ref{thm: contraction general}). In particular, it shows that a uniformly continuous mapping to a complete metric space  transfers a metric LDP into a standard one if the sub-level sets of the rate function are totally bounded (Corollary~\ref{cor: contraction complete}). This allows us to transform the metric LDP of~\cite{BorovkovMogulskii2, BorovkovMogulskii3} into standard LDPs with tight rate functions (Theorem~\ref{thm: contraction}). For example, this yielded a standard LDP (Proposition~\ref{prop: LDP weak}) for scaled trajectories of random walks in $\R^d$ considered as random elements of the space $BV[0,1]$ of functions of bounded variation equipped with the metric $\rho_*$ of joint convergence of functions in $L^1$ and of their values at~$1$. It metrizes weak-* convergence on sets of functions of uniformly bounded variation. It is in a certain sense shorter than the Skorokhod metric generating  the topology $M_1$ (Lemma~\ref{lem: metrics}). 

We  used ideas from calculus of variations, which offers well-developed methods for working with integral (action) functionals on the space $BV[0,1]$ equipped with the weak-* topology. In particular, this allowed us to find an explicit {\it integral representation} (Theorem~\ref{thm: I_D}) for the rate function of Borovkov and Mogulskii~\cite{BorovkovMogulskii2, BorovkovMogulskii3}, who gave such representation only in dimension one. 
We also found a wide class of sets where the upper bound in the metric LDP of~\cite{BorovkovMogulskii2, BorovkovMogulskii3} coincides with the standard one (see Propositions~\ref{prop: upper general} and~\ref{prop: upper} and Remark~\ref{rem: closures}).  

We have the following applications of our results on the trajectories of random walks. 

First, from our contraction principle for the trajectories  (i.e.\ Theorem~\ref{thm: contraction}) we obtained the LDPs for the {\it perimeter} and the {\it area} of the {\it convex hull} of a planar random walk, presented in our different paper~\cite[Proposition~4.1]{AkopyanVysotsky}. This application motivated our initial interest in the questions considered here.

Second, from our LDP for the trajectories in $(BV[0,1], \rho_*)$ (i.e.\ Proposition~\ref{prop: LDP weak}), we obtained the LDP for {\it kernel-weighted sums} of i.i.d.\ random vectors in $\R^d$ (Theorem~\ref{thm: weighted}). This extends the classical Cram\'er theorem, without any additional assumptions. We give an explicit formula for the rate function, which is especially accessible in dimension one. To the best of our knowledge, the strongest available results in this direction are by Gantert et al.~\cite{Gantert+} and Kiesel and Stadtm{\"u}ller~\cite{KieselStadtmuller2000}, both works concerning dimension one only. The paper~\cite{KieselStadtmuller2000} considers the ``light-tailed'' case where the i.i.d.\ terms have finite Laplace transform. The paper~\cite{Gantert+} considers  the ``heavy-tailed'' case, where the common distribution of the terms has ``stretched''  exponential tails and thus does not satisfy the Cram\'er moment assumption. Our Theorem~\ref{thm: weighted} complements the one-dimensional results of~\cite{Gantert+},~\cite{KieselStadtmuller2000} and shows that there is a natural transition of the rate function from the  ``light-tailed'' to the ``heavy-tailed'' cases; see also Remark~\ref{rem: weighted}.\ref{item: minimizer}. 
%The methods of the papers~\cite{Gantert+} and~\cite{KieselStadtmuller2000} %are entirely different from the one used here.

\medskip

When the current  paper was ready for submission, we became aware of the work of Najim~\cite{Najim}. His main result, Theorem 3.1 on the LDP for weighted sums of i.i.d.\ random vectors, is stronger than the LDP in our Theorem~\ref{thm: weighted}. Our new contributions are the simplified explicit expressions for the rate function (formulas~\eqref{eq: weight rate f explicit} and~\eqref{eq: weight rate f 1-D}) and for the  corresponding minimizing trajectories we found in certain cases  (Remark~\ref{rem: weighted}.\ref{item: minimizer}). As a corollary to his main result,  J.\ Najim obtained an LDP in the weak-* topology on $BV[0,1]$ (\cite[Theorems~4.1 and~4.3]{Najim}) for trajectories of random walks under the Cram\'er moment assumption. This result is very similar (but not equivalent, cf.~Section~\ref{sec: weak-* and rho_*}) to our Proposition~\ref{prop: LDP weak}.  

Notably, we obtained our extension of Cram\'er's theorem as a corollary and only at the last step of our work, as opposed to the argument of~\cite{Najim} going in the reverse direction.
%, a priori it was completely unclear that one should search in this direction %to avoid rediscovering the results already available. 
Therefore, despite of  a significant intersection with~\cite{Najim}, we believe that the corresponding part of our paper is of independent interest since it has a different motivation and uses a different approach -- we used Proposition~\ref{prop: LDP weak}, which itself follows from the metric LDP in~\cite{BorovkovMogulskii2, BorovkovMogulskii3} using the metrization result of Theorem~\ref{thma: Hognas}, the contraction principle in Theorem~\ref{thm: contraction general}, and geometric Lemma~\ref{lem: metrics}, which is of its own interest. 
%The contraction principle and the results clarifying the relation between the %non-standard and standard LDPs stress distinction of our work from~\cite{Najim}. 

\medskip

This paper is organized as follows. Section~\ref{sec: prelim} presents the metric LDP for trajectories of random walks by Borovkov and Mogulskii~\cite{BorovkovMogulskii2, BorovkovMogulskii3} and also introduces metric LDPs in general metric spaces. In Section~\ref{Sec: contraction} we prove a contraction principle for uniformly continuous mappings of metric LDPs, and present its application for random walk trajectories. In Section~\ref{sec: weak-*} we define the weak-* and related topologies on the space of functions of bounded variation, and compare them with and the Skorokhod topologies $M_1$ and $M_2$ and their modifications. 
%and the closely connected Jakubowki topology $S$ (which proved to be quite popular but has no use in this paper). 
The main result of the section is geometric Lemma~\ref{lem: metrics}, which compares the metric $\rho_*$ with Skorokhod metrics. The setup and results of this section are used in Section~\ref{sec: integral}, where we give an integral representation for the rate function, prove an LDP for the trajectories in the space $(BV[0,1], \rho_*)$, and present a few types of sets where the upper bound in the metric LDP for the trajectories equals the standard one. The last section concerns our main application, Cram\'er's theorem for kernel-weighted sums of i.i.d.\ random vectors in~$\R^d$.

\section{Metric LDPs for trajectories of random walks} \label{sec: prelim}
In this section we give necessary definitions and provide a brief  summary of the results of Borovkov and Mogulskii~\cite{BorovkovMogulskii2, BorovkovMogulskii3}.

\subsection{General metric LDPs} \label{sec: metric LDP}
We start with general definitions. Let $\mathcal X$ be a Hausdorff topological space equipped with the Borel $\sigma$-algebra. Let $\mathcal{I}:\mathcal X\to[0,+\infty]$ be a lower semi-continuous function such that $\mathcal I \not \equiv +\infty$. By definition, this means that the {\it sub-level sets} $\{x\in \mathcal X \,:\, \mathcal{I}(x) \leq \alpha \}_{\alpha \in [0,  \infty)}$ of $\mathcal I$ are closed. If $\mathcal X$ is a metric space, this is equivalent to requiring that for every $x \in \mathcal X$, we have $\mathcal I(x) \le \liminf_{n \to \infty} \mathcal I(x_n)$ for any sequence $(x_n)_{n \ge 1}$ converging to $x$. 

We say that a sequence $(Z_n)_{n \ge 1}$  of random elements of $\mathcal X$ satisfies the {\it large deviations principle (LDP)} in $\mathcal X$ with a {\it rate function} $\mathcal{I}$ and a {\it speed} $(a_n)_{n \ge 1} \subset (0, \infty)$ if for every Borel set $B\subset \mathcal X$,
\begin{equation}\label{eq: LD general}
-\inf_{x\in \intr B}\mathcal I(x) \leq\liminf_{n\to\infty} \frac{1}{a_n}\log \P(Z_n\in B)  \leq\limsup_{n\to\infty}\frac{1}{a_n}\log \P(Z_n\in B) \leq-\inf_{x\in \cl B}\mathcal I(x),
\end{equation}
where $\inf_\varnothing := +\infty$ by the usual convention. 
%Likewise, $(Z_n)_{n \ge 1}$ satisfies a {\it weak LDP} if the lower bound in \eqref{eq: LD general} holds for every $B$ and the upper bound is satisfied only if the set $\cl B$ is compact. 
If $\mathcal X$ is equipped with a metric $\ell$, we say that  $(Z_n)_{n \ge 1}$ satisfies the {\it metric LDP}  in $(\mathcal X, \ell)$  with a  rate function  $\mathcal{I}$ and a  speed $(a_n)_{n \ge 1}$~if 
\begin{equation}\label{eq: LD extended}
-\inf_{x\in \intr B}\mathcal I(x) \leq\liminf_{n\to\infty}\frac{1}{a_n}\log \P(Z_n\in B)  \leq\limsup_{n\to\infty}\frac{1}{a_n}\log \P(Z_n\in B) \leq- \lim_{\varepsilon \to 0+} \inf_{x \in B^\varepsilon_\ell} \mathcal I(x),
\end{equation}
where $B^\varepsilon_\ell$ denotes the open $\varepsilon$-neighbourhood of $B$ in $\ell$ (with the convention $\varnothing^\varepsilon_\ell:= \varnothing$). For brevity, we will not refer to the speed when $a_n=n$, which essentially is the only case of our interest in this paper. We stress that the concept of metric LDPs is not a  topological one, in the sense that~\eqref{eq: LD extended} may cease to hold when $\ell$ is replaced by another metric generating the same topology. 

We say that the rate function $\mathcal I$ is  {\it tight} if its  sub-level sets  are compact.  A natural way to relax this condition in the case when $\mathcal X$ is a metric space is to assume total boundedness of the sub-level sets of $\mathcal I$. Recall that a subset of a metric space is {\it totally bounded} if it has a finite $\varepsilon$-net for every $\varepsilon>0$; such subset is relatively compact, i.e.\ its closure is compact, if the metric space is {\it complete}.

Note that if $\mathcal I$ is tight, then \eqref{eq: LD general} and \eqref{eq: LD extended} coincide because 
\begin{equation} \label{eq: tight blow up}
\inf_{x \in \cl B} \mathcal I(x) = \lim_{\varepsilon \to 0+} \inf_{x \in B^\varepsilon_\ell} \mathcal I(x);
\end{equation}
see~\cite[Lemma~1.1]{BorovkovMogulskii1} or Dembo and Zeitouni~\cite[Lemma~4.1.6(b)]{DemboZeitouni}.
%; moreover,  the last infimum in \eqref{eq: LD general} is always attained at some~$x \in \cl B$ unless $B$ is empty. 
It is also easy to see that this equality holds true for any $\mathcal I$  when $B$ is relatively compact.

Equality \eqref{eq: tight blow up} may not hold in general (see Remark~\ref{rem}.\ref{item: tight blow up fails}), and thus
a metric LDP is in general strictly weaker than the corresponding standard one.  For an example when a metric LDP holds true but the standard one does not, see Bazhba et al.~\cite[Theorem~2.1 and Section~3.3]{Bazhba+}. On the other hand, a metric LDP is stronger than the corresponding {\it weak} LDP, defined as the standard one but with the upper bound required only for relatively compact sets. For this reason, we find it misleading to refer to metric LDPs as {\it extended} ones, which is the original terminology of Borovkov and Mogulskii~\cite{BorovkovMogulskii2010}. Moreover, neither standard nor weak LDPs require a metric structure.
%; perhaps ``expanded'' would be a better term.  

\subsection{Skorokhod topologies} \label{sec: Skorokhod top}
We will write $x = (x^{(1)}, \ldots, x^{(d)})$ for the Cartesian coordinates of $x \in \R^d$ with $d \in \N$, $|x|$ for the Euclidean norm, and `$\cdot$' for the scalar product on $\R^d$.

Denote by $D[0,1]=D([0,1];\R^d)$ the set of {\it c\`adl\`ag} functions on $[0,1]$, that is right-continuous $\R^d$-valued functions  without discontinuities of the second kind\footnote{A function on $[0,1]$ has no discontinuities of the second kind if it has right and left limits at every point.}. We will consider  several Skorokhod-type metrics and topologies on $D[0,1]$. We refer to the book by Whitt~\cite{Whitt}, where Chapter~12 gives a comprehensive treatise of the matter. 

The {\it completed graph} $\Gamma h$ of a function $h \in D[0,1]$ is a subset of $[0,1] \times \R^d$ defined by 
\[
\Gamma h:= \big \{(t, x): 0 \le t \le 1, x \in  [h(t-), h(t+)] \big \},
\]
where $h(0-):=h(0)$, $h(1+):=h(1)$, and $[u_1, u_2]$ denotes the line segment with the endpoints $u_1, u_2 \in \R^d$. 
We deliberately wrote $h(t+)$ instead of the equal quantity $h(t)$, to stress that $\Gamma h$ is defined even when $h$ has no discontinuities of the second kind but is not right-continuous. 
We equip the completed graphs with the topology induced from $[0,1] \times \R^d$.

Consider a set of parametrizations of the completed graph:
$$
\Pi(h):=\big\{\gamma
\big| \big. \gamma:[0,1] \to \Gamma h \text{ is bijective, continuous, and satisfying } \gamma(0)=(0, h(0)) \big\}
$$
(this set is non-empty by \cite[Remark~12.3.3]{Whitt}). We can thus regard the completed graphs as images of continuous curves in $\R^d$. The metric $\rho_1$ on $D[0,1]$ is  the least uniform distance between parametrizations of completed graphs:
\begin{equation} \label{eq: rho parametric}
\rho_1(h_1, h_2):=\inf_{\gamma_1 \in \Pi(h_1), \gamma_2 \in \Pi(h_2)} \sup_{t \in [0,1]} |\gamma_1(t) - \gamma_2(t)|.
%, \qquad h_1, h_2 \in D[0,1].
\end{equation}
%(and the topologies on the graphs $\Gamma h_i$ are induced from $[0,1]\times \R^d$).
%In particular, $\gamma_i^{(1)}$ are non-decreasing mappings of $[0,1]$ onto itself. 
The topology generated by $\rho_1$ is called the {\it Skorokhod topology} $M_1$; see~\cite[Remark~12.3.4]{Whitt}. The metric $\rho_2$ on $D[0,1]$ is the Hausdorff distance $d_H$ between the completed graphs, i.e.\ $\rho_2(h_1, h_2):= d_H(\Gamma h_1, \Gamma h_2)$. In other words,
\begin{equation} \label{eq: rho Hausd}
\rho_2(h_1, h_2) =  \max_{i =1,2 }\max_{(s,x) \in \Gamma h_i} \min_{(t,y) \in \Gamma h_{3-i}} |(s, x) - (t, y)|.
%, \qquad h_1, h_2 \in D[0,1].
\end{equation}
This is a genuine distance because each $\Gamma h$ is compact, being a continuous image of $[0,1]$.
%, and $\Gamma h_1 = \Gamma h_2$ implies $h_1=h_2$ for $h_1,h_2\in D[0,1]$. 
The topology generated by $\rho_2$ is called the {\it Skorokhod topology} $M_2$; see~\cite[Theorem~12.10.1]{Whitt}.

Furthermore, consider the modified completed graphs $\Gamma' h:=\Gamma h \cup [0, h(0)]$ and define $\rho_1'$ and $\rho_2'$ exactly as above using $\Gamma' $ (parametrized by functions in $\Pi'$) instead of~$\Gamma$. Then $\rho_1'$, $\rho_2'$ are metrics too.  
%(the triangle inequality for $\rho_1'$ follows from the fact that for any $h %\in D[0,1]$ and $\gamma_1, \gamma_2 \in \Pi'(h)$, there exists an %increasing bijection $\lambda:[0,1] \to [0,1]$ such that $\gamma_2 = %\gamma_1 \circ \lambda$). 
Equivalently, for $h_1, h_2 \in D[0,1]$ we can write 
\begin{equation} \label{eq: rho and rho'}
\rho_i'(h_1, h_2)= \rho_i(h_1 \I_{(0,1]}, h_2 \I_{(0,1]})
\end{equation}
if we extend definitions \eqref{eq: rho parametric} and~\eqref{eq: rho Hausd} to the space of functions on $[0,1]$ that are c\`adl\`ag  on $(0,1]$ and have right limit at $0$, where the $\rho_i$'s still remain metrics because completed graphs uniquely define such functions.

From the definitions above, we readily obtain the following inequalities on $D[0,1]$:
\begin{equation} \label{eq: Hausd shorter}
\rho_2' \le \rho_1' \le \rho_1 \quad \text{and} \quad \rho_2 \le \rho_1.
\end{equation}
The inequality $\rho_1' \le \rho_1$ follows because $\rho_1(h_1, h_2)$ equals the infimum in the definition of $\rho_1'$ taken over the pairs of parametrizations $\gamma_1 \in \Pi'(h_1)$, $\gamma_2 \in \Pi'(h_2)$ satisfying $\gamma_1([0,t])=[0, h_1(0)]$ and $\gamma_2([0,t])=[0, h_2(0)]$ for some $t \in [0,1]$. Moreover, we have
\begin{align} \label{eq: Hausd' shorter}
\rho_2'(h_1, h_2) &= d_H\big(\Gamma  h_1 \cup [0, h_1(0)], \Gamma  h_2 \cup [0, h_2(0)] \big) \notag \\
&\le \max \Big(d_H(\Gamma  h_1, \Gamma  h_2), d_H \big([0, h_1(0)], [0, h_2(0)] \big) \Big) \notag \\
&\le \max\big(\rho_2(h_1, h_2), |h_1(0)-h_2(0)| \big), 
\end{align}
where we used the fact that $d_H(A_1 \cup B_1, A_2 \cup B_2) \le \max(d_H(A_1, A_2), d_H(B_1, B_2))$ for any non-empty $A_1, A_2, B_1, B_2 \subset \R^{d+1}$ (which follows easily from the definition of $d_H$).

We will work with the topologies $M_1'$ and $M_2'$ on $D[0,1]$ generated respectively by $\rho_1'$ and~$\rho_2'$. The topology $M_1'$ was recently used by Bazhba et al.~\cite{Bazhba+}; the versions of $M_1'$ and $M_2'$ on $D[0, \infty)$ briefly appeared in~\cite[Section~13.6.2]{Whitt}. 
Inequalities \eqref{eq: Hausd shorter} and \eqref{eq: Hausd' shorter} imply that
\begin{equation} \label{eq: M inclusions}
M_2' \subset M_1' \subset M_1 \quad \text{and} \quad M_2' \subset M_2 \subset M_1,
\end{equation}
where the third inclusion follows from \eqref{eq: Hausd' shorter} using that every $M_2'$-closed set is also closed in $M_2$ because convergence of functions in $M_2$ implies convergence of their values at~$0$. 

The following result, presented without a proof because we will not use it in this paper, describes convergence in $M_i'$ in terms of more standard convergence in $M_i$. It shows that naturally, the difference between the convergences is only in the behaviour around time $0$. In particular, the value  at $0$ is not an $M_i'$-continuous functional.

\begin{prop} \label{prop: M_i and M_i'}
Let $h, h_1, h_2, \ldots \in D[0,1]$ and $i \in \{1,2\}$. Then $\lim_{n \to\infty }\rho_i'(h_n, h)=0$ if and only if  there exits a sequence $t_1, t_2, \ldots \in [0,1]$ such that $\lim_{n \to \infty} t_n=0$,
\[
\lim_{n \to\infty }\rho_i(h_n(\cdot \vee t_n), h)=0,
\] 
and
\[
\begin{cases}
\lim_{n \to \infty}\sup_{0 \le t \le t_n}   \sup_{0 \le s \le t} \big (|h_n(s)| \cdot |h(0)| - h_n(t) \cdot h(0) \big ) =0, & \text{ if } i=1,\\
\lim_{n \to \infty}\sup_{0 \le t \le t_n} \min_{0 \le s \le 1} |s h(0)-h_n(t)| = 0, & \text{ if } i=2.
\end{cases}
\]
\end{prop} 
The last condition means that whenever $h(0) \neq 0$, for $i=1$ the values of $h_n$ on $[0, t_n]$  are nearly proportional to $h(0)$ and their orthogonal projections on $h(0)$ are nearly non-decreasing, and for $i=2$ they nearly belong to $[0, h(0)]$. 

Lastly, we note that each of the metrics $\rho_i$, $ \rho_i'$ is separable and incomplete. Separability follows from Lemma 1 in Section~14 in Billingsley~\cite{Billingsley}. Observing that $\{ \I_{[1/2, 1/2+1/n)} \}_{n \ge 2}$ is a Cauchy sequence in each of the metrics which does not converge shows their incompleteness.

\subsection{Metric LDPs for trajectories of random walks} \label{sec: BorMog}
Let $(S_n)_{n \ge 1}$, where $S_n = X_1 + \ldots + X_n$, be a random walk with independent identically distributed increments $X_1, X_2, \ldots$ in $\R^d$, where $d \ge 1$. For any $n \in \N$, let $S_n(\cdot)$ be the piece-wise linear function on $[0,1]$ defined by linear interpolation between its values at the points $k/n$, where $0 \le k \le n$, $k \in \Z$, that are given by $S_n(k/n):=S_k$, where $S_0:=0$. These are time-rescaled trajectories of the random walk $(S_n)_{n \ge 1}$. We will regard them as random elements of the space $D[0,1]$ (or its subsets) equipped with the Borel $\sigma$-algebras generated by $\rho_1'$ or $\rho_2'$.

Let $\mathcal L(u):= \E e^{u \cdot X_1}$, where $u \in \R^d$, be the Laplace transform of the random vector $X_1$ in $\R^d$. Denote by $\mathcal D_{\mathcal L}:=\{u \in \R^d: \mathcal L(u)< \infty\}$ the {\it effective domain} of $\mathcal L$. We say that $X_1$ satisfies the {\it Cram\'er moment assumption} if $\mathcal L$ is finite in an open neighbourhood of $0$, that is $0 \in \intr \mathcal D_{\mathcal L}$ in short. The function $K:=\log \mathcal L$, called the {\it cumulant moment generating function} of $X_1$, is always convex. Denote by $I$ the {\it Legendre--Fenchel transform} of $K$, i.e.
\begin{equation} \label{eq: Legendre-Fenchel}
I(v):=\sup_{u \in \R^d} \big( u \cdot v - K(u) \big), \quad v \in \R^d.
\end{equation}
This is a convex lower semi-continuous function with values in $[0, \infty]$. The classical Cram\'er theorem states that under the assumption $0 \in \intr \mathcal D_{\mathcal L}$, the sequence $(S_n/n)_{n \ge 1}$ satisfies the LDP in the Euclidean space $\R^d$ with the tight rate function~$I$. This justifies saying that $I$ is the {\it rate function} of~$X_1$.

The {\it total variation} $\Var(h)$ of a function $h \in D[0,1]$ is defined by
\begin{equation} \label{eq: Var def}
\Var(h):= \sup_{\mathbf t \subset (0,1]: \, \# \mathbf t < \infty } \int_0^1 |(h^{\mathbf t})'(s)| ds, 
\end{equation}
where $h^{\mathbf t}$ denotes the continuous function on $[0,1]$ defined by linear interpolation between its values $\mathbf t \cup \{0,1\}$ that are given by $h^{\mathbf t}(s) := h(s)$ for $s \in \mathbf t \cup \{1\}$ and $h^{\mathbf t}(0):=0$. We can regard $\Var(h)$ as the length of the curve in  $\R^d$ obtained by taking the spatial coordinate of a curve in $\Pi'(h)$. Denote by $BV[0,1]:=\{h \in D[0,1]: \Var(h)< \infty\}$ the set of c\`adl\`ag $\R^d$-valued functions of {\it bounded variation}, and by $AC_0[0,1]$ its subset of coordinate-wise {\it absolutely continuous} functions such that $h(0)=0$. 

Lastly, define a non-negative functional $I_D$ on $h \in D[0,1]$:
\begin{equation} \label{eq: I_D def}
I_D(h):= \sup_{\mathbf t \subset (0,1]: \, \# \mathbf t < \infty } \int_0^1 I((h^{\mathbf t})'(s)) ds. 
\end{equation}
It is worth noting that $I_D(h)=\int_0^1 I(h'(s))ds$ when $h \in AC_0[0,1]$ (\cite[Theorem~5.3]{BorovkovMogulskii2}).

We now present the metric LDP for trajectories of random walks.

\begin{theorema}[Borovkov and Mogulskii~\cite{BorovkovMogulskii2, BorovkovMogulskii3}] \label{thma: BM}
Assume that $X_1$ is a random vector in $\R^d$ such that $0 \in \intr \mathcal D_{\mathcal L}$. Then the  random sequences  $(S_n(\cdot)/n)_{n \ge 1}$ and  $(S_{[n \, \cdot]}/n)_{n \ge 1}$ satisfy the metric LDPs~\eqref{eq: LD extended} in each of the four metric spaces $(D[0,1], \rho_i')$ and $(BV[0,1], \rho_i')$ for $i\in \{1,2\}$, with the rate function $I_D$ whose sub-level sets are totally bounded (in each space). 

Moreover, $I_D$ is convex  and it satisfies, for some constants $c_1, c_2 >0$,
\begin{equation} \label{eq: I_D bound}
I_D(h) \ge c_1 \Var (h) - c_2, \qquad h \in D[0,1].
\end{equation}
\end{theorema}

\begin{remark} \label{rem} 
Let us make a number of comments. 
\begin{enumerate}
\item \label{item: I_D vanish} Note that $I_D(h)=+\infty$ for $h \not \in BV[0,1]$ by \eqref{eq: I_D bound}. 

\item  The metric LDPs for $S_n(\cdot)$ and $S_{[n \, \cdot]}$ are equivalent by  $\rho_i(S_n(\cdot)/n, S_{[n \, \cdot]}/n) \le 2/n$.  The metric LDPs in $\rho_1'$ are stronger than the ones in $\rho_2'$ since $\rho_2' \le \rho_1'$. We stated these weaker results to match the presentation of~\cite{BorovkovMogulskii2, BorovkovMogulskii3}, which puts emphasis on the metric $\rho_2'$. The only advantage of $\rho_2'$ is in its relative simplicity. 

\item Although the rate function $I_D$ is not tight, equality \eqref{eq: tight blow up} still holds true with $\mathcal I =I_D$ and $\ell=\rho_1'$ for certain types of non-relatively compact sets $B$, described in Remark~\ref{rem: closures}. 

\item \label{item: tight blow up fails} In general, \eqref{eq: tight blow up} does not hold for $\ell=\rho_2'$. For example, assume that $d=1$, $\E X_1 =0$, $\mathcal D_{\mathcal L}$ is bounded, and consider $B := \{h_n\}_{n \ge 4}$ with $h_n:=\I_{[1/2-2/n,1/2-1/n) \cup [1/2+1/n, 1/2+2/n)}$. No subsequence $\{h_{n_k}\}$ converges in $\rho_2'$ to an element of $D[0,1]$, hence $B$ is closed w.r.t.\ $\rho_2'$. On the other hand, we have $\rho_2'(g_n,B)=1/n$ for $g_n:=\I_{[1/2-1/n,1/2+1/n)}$. Then $I_D(g_n)=\frac12 I_D(h_n) = const >0$  by equality \eqref{eq: I_D 1-D} below, therefore \eqref{eq: tight blow up} cannot hold.

It is plausible that \eqref{eq: tight blow up} does not hold for $\ell=\rho_1'$ too but we have no examples. The papers~\cite{BorovkovMogulskii2010, BorovkovMogulskii1, BorovkovMogulskii2, BorovkovMogulskii3} offer no discussion on this question.

\item The reason why the metric LDPs do not immediately imply the corresponding LDPs is incompleteness of the metric spaces considered.  There exists an explicit complete metric that generates the topology $M_1$ (\cite[Section 12.8]{Whitt}), and  it appears that its minor modification  should give a complete metric $\tilde \rho_1'$ generating $M_1'$. However, such complete metrics are longer than the initial ones, therefore it seems impossible to have the upper bound in~\eqref{eq: LD extended} with $\ell = \tilde \rho_1'$ instead of $\ell = \rho_1'$.

\item  Bound~\eqref{eq: I_D bound} readily follows from the inequality $I(v) \ge c_1 |v| - c_2$ for $v \in \R^d$, which holds true because $I$ is convex and grows at least linearly at infinity; see \eqref{eq: support function} and \eqref{eq: recession function} below.

\item In view of \eqref{eq: rho and rho'} and given that $S_n(0)=S_0=0$, one may argue that it would be more natural to employ the non-standard space $D'[0,1]:=\{h \I_{(0,1] }: h \in D[0,1]\}$ equipped with the usual Skorokhod metrics $\rho_i$. This is essentially done in \cite{BorovkovMogulskii2, BorovkovMogulskii3}. The spaces $(D'[0,1], \rho_i)$ and $(D[0,1], \rho_i')$ are isometric, so the metric LDPs transfer easily. On the other hand, it is natural to work with the space $BV[0,1]$ of c\`adl\`ag modifications of functions of bounded variation because these are distribution functions of vector-valued finite measures. This explains our choice of the standard space $D[0,1]$.

\item The ultimate reason why the space $BV[0,1]$ arises is that the rate function $I$ is not super-linear at infinity. The only exception is when the Laplace transform of the increments is finite (see~\eqref{eq: support function} and~\eqref{eq: recession function}), in which case it suffices to work with the space $(AC_0[0,1], \| \cdot \|_\infty)$ and the rate function $I_D$ is tight due to the super-linearity of $I$. Such effects of the behaviour of integrand at infinity are well-known in the calculus of variations, which studies minimization of integral functionals.
\end{enumerate}
\end{remark}

Theorem~\ref{thma: BM} is a combination and adaptation of several results scattered through~\cite{BorovkovMogulskii1, BorovkovMogulskii2, BorovkovMogulskii3}, therefore we shall explain in detail how  we obtained it. The authors of~\cite{BorovkovMogulskii2, BorovkovMogulskii3} considered a wider space $\mathbb D$ of functions $h:[0,1] \to \R^d$ without discontinuities of the second kind satisfying $h(t) \in [h(t-), h(t+)]$ for $t \in [0,1]$. The functions $\rho_i$, defined on $\mathbb D$ as above in~\eqref{eq: rho parametric} and~\eqref{eq: rho Hausd}, are now pseudometrics. Put $\mathbb D_0:=\{h \in \mathbb D: h(0)=0\}$ and define the functionals $\Var$ and $I_D$ on $\mathbb D_0$ as above in~\eqref{eq: Var def} and~\eqref{eq: I_D def}. The functional $I_{\mathbb D}$ on $\mathbb D$ is defined by $I_{\mathbb D}:=I_D$ on $\mathbb D_0$ and $I_{\mathbb D}:=+\infty$ on $\mathbb D_0^c$, see~\cite[Definition~2.1]{BorovkovMogulskii1}. Lastly, putting $h^+(t):=h(t+)$ for $t \in [0,1]$ defines an isometry from $(\mathbb D_0, \rho_i)$ onto $(D[0,1], \rho_i')$.

Then $(S_{[n \, \cdot]}/n)_{n \ge 1}$ satisfies the metric LDP in $(\mathbb D, \rho_2)$ with convex rate function $I_{\mathbb D}$ by \cite[Theorem~5.5]{BorovkovMogulskii2}, where measurability refers to the Borel $\sigma$-algebra (see~\cite[Definition~1.4]{BorovkovMogulskii1}),
%combined with~\cite[Lemma~3.1]{BorovkovMogulskii2} ensuring that the Borel $\sigma$-algebra on $(\mathbb D, \rho_2)$ is contained in the one generated by the cylinder sets on $[0,1]$
and $I_{\mathbb D}$ is convex and lower semi-continuous  by~\cite[Theorem~5.2.(i),(ii)]{BorovkovMogulskii2} and its sub-level sets are totally bounded by~\cite[Lemma~5.3]{BorovkovMogulskii2}  combined with inequality \eqref{eq: I_D bound}. By Lemma~\ref{lem: inclusions}.\ref{item: restrict}, the metric LDP remains valid on $(\mathbb D_0, \rho_2)$ because $\P(S_{[n \, \cdot]} \in \mathbb D_0)=1$ for every $n$ and $I_{\mathbb D}=+\infty$ on $\mathbb D_0^c$. Finally, by our Theorem~\ref{thm: contraction general}, the isometry $h \mapsto h^+$ transforms this metric LDP into the one on $(D[0,1], \rho_2')$ with the rate function $I_D$ because $S_{[n \, \cdot]}^+ = S_{[n \, \cdot]}$ and $I_{\mathbb D}(h)=I_D(h^+)$ for every $h \in \mathbb D_0$ by \cite[Theorem 5.1]{BorovkovMogulskii2}.

Denote by $\mathbb V_0$ the subset of $\mathbb D_0$ of functions of finite variation. By~\cite[Theorem 6.2]{BorovkovMogulskii3}, $(S_{[n \, \cdot]}/n)_{n \ge 1}$ satisfies the metric LDP in $(\mathbb V_0 , \rho_1)$ with the rate function $I_{\mathbb D}$, which is lower semi-continuous (in $\rho_1$) because it is so in the shorter pseudometric $\rho_2$, and its sub-level sets are totally bounded  by~\cite[Lemma~6.2]{BorovkovMogulskii3}  combined with~\eqref{eq: I_D bound}. This metric LDP remains valid in $(BV[0,1], \rho_1')$ by the same argument as above using Remark~\ref{rem}.\ref{item: I_D vanish}. It in turn implies the weaker metric LDP in $(BV[0,1], \rho_2')$ and also implies the one in $(D[0,1], \rho_1')$ by Lemma~\ref{lem: inclusions}.\ref{item: extend} using that $I_D$ is lower semi-continuous in $\rho_1'$ because it is so in $\rho_2'$.

\section{Contraction principle for metric LDPs} \label{Sec: contraction}

The following general result is analogous to the usual contraction principle for standard LDPs; cf.~\cite[Theorem~4.2.1 and Remark~(c)]{DemboZeitouni}. To state it, we first give two definitions.

For a function $J : \mathcal X \to [0, +\infty]$ defined on a topological space $\mathcal X$, denote by $\cl J$ its {\it closure} (or the {\it lower semi-continuous regularization}), i.e.\ the function whose epigraph is the closure (in the product topology on $\mathcal X \times [0, +\infty]$) of the epigraph  of $J$. Recall that $\mathcal D_J=\{x \in \mathcal X: J(x)<\infty\}$ denotes the effective domain of $J$. We say that a mapping $F$ between metric spaces is uniformly continuous {\it on a subset} $A$ of the domain if $F|_A$ is uniformly continuous. 

\begin{theorem} \label{thm: contraction general}
Let $(\mathcal X, \ell_1)$ and $(\mathcal Y, \ell_2)$ be metric spaces and $(Z_n)_{n \ge 1}$ be a sequence of random elements  that satisfies a metric LDP in $(\mathcal X, \ell_1)$ with some speed and a rate function $\mathcal I$. Let $F:\mathcal X \to \mathcal Y$ be a measurable mapping that is continuous at every $x \in \mathcal D_{\mathcal I}$ and  uniformly continuous on every sub-level set of $\mathcal I$. Then the sequence $(F(Z_n))_{n \ge 1}$ satisfies the metric LDP in $(\mathcal Y, \ell_2)$ with the same speed and the rate function $\cl \tilde{\mathcal J}$, where $\tilde{\mathcal J}(y):= \inf_{x \in F^{-1}(y)} \mathcal I(x)$ for $y \in \mathcal Y$. 

Moreover, if  the sub-level sets of $\mathcal I$ are totally bounded, then the same is true for $\cl \tilde{\mathcal J}$.  
\end{theorem}

The interest in this result is in its corollary, which allows one to bring the metric LDPs~\eqref{eq: LD extended} into the standard form~\eqref{eq: LD general}. 

\begin{cor} \label{cor: contraction complete}
If the metric space $(\mathcal Y, \ell_2)$ is complete and the sub-level sets of $\mathcal I$ are totally bounded, then the sequence $(F(Z_n))_{n \ge 1}$ satisfies the (standard) LDP with the tight rate function $\cl \tilde{\mathcal J}$.
\end{cor}

This follows from Theorem~\ref{thm: contraction general} by equality \eqref{eq: tight blow up} and the fact that closed totally bounded subsets of complete metric spaces are compact. 

Our main application of Corollary~\ref{cor: contraction complete} is in the context of random walks trajectories:

\begin{theorem} \label{thm: contraction}
Assume that $X_1$ is a random vector in $\R^d$  such that $0 \in \intr \mathcal D_{\mathcal L}$.  Let $\mathcal Y$ be a complete metric space
and $F: BV[0,1] \to \mathcal Y$ be a mapping that is continuous in $\rho_1'$ and uniformly continuous in $\rho_1'$ on $\{h: \Var(h) \le R\}$ for every $R>0$. Then both sequences of random elements $(F(S_n(\cdot)/n))_{n \ge 1}$ and $(F(S_{[n \cdot ]}/n))_{n \ge 1}$ satisfy the (standard) LDP in $\mathcal Y$ with the tight rate function   $\cl \tilde{\mathcal J}$, where $\tilde{\mathcal J}(y):= \inf_{h \in F^{-1}(y)} I_D(h)$ for $y \in \mathcal Y$.
\end{theorem}

This follows from Corollary~\ref{cor: contraction complete} combined with Theorem~\ref{thma: BM} using  the lower bound \eqref{eq: I_D bound} for $I_D(h)$ in terms of $\Var(h)$.

\begin{proof}[{\bf Proof of Theorem~\ref{thm: contraction general}.}]
For any $A \subset \mathcal Y$, we have
\[
\{ x: \mathcal I(x) < \infty \}   \cap  F^{-1}(\intr A)  =  \{ x: \mathcal I(x) < \infty \} \cap  \intr (F^{-1}(\intr A))
\]
because for every $x$ in the set in the l.h.s.,  $\{x\}_{\ell_1}^\delta  \subset F^{-1}(\intr A) $ holds for some $\delta > 0$ by continuity of $F$ at $x$, and thus $x \in \intr( F^{-1}(\intr A))$. Hence
\begin{equation} \label{eq: change 2}
\inf_{x \in \intr( F^{-1}(\intr A)) } \mathcal I(x) = \inf_{x \in  F^{-1}(\intr A)} \mathcal I(x)  = \inf_{y \in  \intr A} \tilde{\mathcal J}(y),
\end{equation}
where the second equality holds true by the definition of $\tilde{\mathcal J}$. Furthermore, we claim that 
\begin{equation} \label{eq: change 1}
\lim_{\delta \to 0+} \inf_{x \in (F^{-1}(A))^\delta_{\ell_1}} \mathcal I (x) \ge \lim_{\varepsilon \to 0+} \inf_{x \in F^{-1}(A^\varepsilon_{\ell_2})} \mathcal I(x) = \lim_{\varepsilon \to 0+} \inf_{y \in A^\varepsilon_{\ell_2} }  \tilde{\mathcal J}(y).
\end{equation}

The inequality is trivial when its l.h.s.\ is infinite, otherwise denote the l.h.s.\ by $R$.
By uniform continuity of $F$ on sub-level sets of $\mathcal I$, for any $\varepsilon >0$ there exists a $\delta_0 >0$ such that $\ell_2(F(x),F(x'))<\varepsilon$ whenever $\ell_1(x, x')< \delta_0$ and $\max(\mathcal I(x), \mathcal I(x')) \le R$. Hence 
\[
\big( F^{-1}( A) \cap \{ x: \mathcal I(x) \le R\}\big)^{\delta_0}_{\ell_1} \subset F^{-1}(A^\varepsilon_{\ell_2}),
\]
and for any $\delta\in(0, \delta_0)$ we have
\[
R \ge \inf_{x \in (F^{-1}(A))^{\delta_0}_{\ell_1}} \mathcal I (x) = \inf_{x \in (F^{-1}(A))^{\delta_0}_{\ell_1} \cap \{ \mathcal I \le R\}^{\delta_0}_{\ell_1}} \mathcal I (x) = \inf_{x \in (F^{-1}(A)  \cap \{ \mathcal I \le R\})^{\delta_0}_{\ell_1}} \mathcal I (x) \ge \inf_{x \in F^{-1}(A^\varepsilon_{\ell_2})} \mathcal I (x),
\]
which implies \eqref{eq: change 1} by first taking $\delta \to 0+$ and then $\varepsilon \to 0+$.

Denoting by  $(a_n)_{n \ge 1}$ the speed in the metric LDP for $(Z_n)_{n \ge 1}$, for any Borel set $A \subset \mathcal Y$,
\begin{align*}
\limsup_{n\to\infty}\frac{1}{a_n}\log \P(F(Z_n)\in A)  &\le - \lim_{\delta \to 0+} \inf_{x \in(F^{-1}(A))^\delta_{\ell_1}} \mathcal I(x) \le - \lim_{\varepsilon \to 0+} \inf_{y \in A^\varepsilon_{\ell_2} }  \tilde{\mathcal J}(y),
\end{align*}
where the second inequality follows from \eqref{eq: change 1}. We also have
\begin{align*}
\liminf_{n\to\infty}\frac{1}{a_n}\log \P(F(Z_n)\in A)  &\ge - \inf_{x \in \intr (F^{-1}(\intr A))} \mathcal I(x) = - \inf_{y \in \intr A}  \tilde{\mathcal J}(y),
\end{align*}
where the equality follows from \eqref{eq: change 2}. 

Finally, we can replace $\tilde{\mathcal J}$ above by the function $\mathcal J$ given by $\mathcal J:= \cl \tilde{\mathcal J}$, which is lower semi-continuous by definition, non-negative, and not identically $ +\infty$. Indeed, we have
\begin{equation} \label{eq: regularization =}
\inf_{y \in \intr A} \mathcal J(y) = \inf_{y \in \intr A} \tilde{\mathcal J}(y), \qquad A\subset \mathcal Y,
\end{equation}
which follows easily from the representation (e.g., see~\cite[Lemma~2.8]{FirasTimo})
\begin{equation} \label{eq: lsc regularization}
\mathcal J(y) = \sup \Big \{ \inf_{z \in U} \tilde{\mathcal J}(z):  y \in U, U \subset \mathcal Y, U  \text{ is open}\Big \}, \qquad y \in \mathcal Y.
\end{equation}
Thus, the sequence $(F(Z_n))_{n \ge 1}$ of random elements of $\mathcal Y$ satisfies the metric LDP in $(\mathcal Y, \ell_2)$ with the rate function $\mathcal J$ and speed $(a_n)_{n \ge 1}$, as stated.

Note that representation \eqref{eq: lsc regularization} also implies that  $\cl \tilde{\mathcal J} \le \tilde{\mathcal J}$ and 
\begin{equation} \label{eq: sub-level sets}
\{y: \mathcal J(y) < \alpha\} \subset \cl \{y: \tilde{\mathcal J}(y) < \alpha  \}, \qquad \alpha  >0. 
\end{equation} 
Indeed, if there is a $y \in \mathcal Y$ such that $\mathcal J(y)< \alpha $ but $y \not \in \cl \{y: \tilde{\mathcal J}(y) < \alpha  \}$, then since $\mathcal Y$ is a metric space, there is an open ball $U$ centred at $y$ that does not intersect with $\cl \{y: \tilde{\mathcal J}(y) < \alpha  \}$. Thus, $\tilde{\mathcal J} \ge \alpha$ on $U$, hence $\mathcal J(y) \ge \alpha$ by  \eqref{eq: lsc regularization}, which is a contradiction.

If $T$ is a totally bounded subset of $ \mathcal X $ and $F$ is uniformly continuous on $T$, then $F(T)$ is totally bounded in $\mathcal Y$. Therefore, if the sub-level sets of $\mathcal I$ are totally bounded in $\mathcal X$, by
$$
\{y: \tilde{ \mathcal J}(y)< \alpha \} = \Big \{y: \inf_{x \in F^{-1}(y)} \mathcal I(x) < \alpha \Big \} \subset F(\{x: \mathcal I (x)<\alpha\}),
$$
the set on the l.h.s.\ is totally bounded in $\mathcal Y$, and so is its closure. Then the sub-level sets of $\mathcal J$ are totally bounded by \eqref{eq: sub-level sets}, as claimed.

\end{proof}

The proof presented actually reveals a wide class of sets where equality~\eqref{eq: tight blow up} holds true and thus the metric LDP bound \eqref{eq: LD extended} can be strengthened to the standard one~\eqref{eq: LD general}. Let us state this as a separate assertion.

\begin{prop} \label{prop: upper general}
Let $(\mathcal X, \ell)$, $(\mathcal Y, \ell_2)$ be metric spaces, and $F: \mathcal X \to \mathcal Y$, $\mathcal I: \mathcal X \to [0, \infty]$ be mappings such that $F$ is uniformly continuous on the sub-level sets of $\mathcal I$. Assume that the function $\tilde{\mathcal J}(y):= \inf_{x \in F^{-1}(y)} \mathcal I(x)$ on $\mathcal Y$ is tight. Then equality~\eqref{eq: tight blow up} holds true for every non-empty set $B $ such that $B=F^{-1}(A)$ for some $A \subset \mathcal Y$ satisfying $\cl (F^{-1}(A))=F^{-1}(\cl A)$, with the infimum on the l.h.s.\ of~\eqref{eq: tight blow up} attained at some~$x \in \cl B$.
\end{prop}

For a direct application of this result to random walks trajectories, see Proposition~\ref{prop: upper}.

\begin{proof}
By \eqref{eq: change 1} and  \eqref{eq: tight blow up} applied to $\mathcal I$ replaced by $\tilde{ \mathcal J}$, which is tight by assumption (and hence lower semi-continuous), we have
\[
\inf_{x \in \cl(F^{-1}(A))} \mathcal I(x) \ge \lim_{\delta \to 0+} \inf_{x \in (F^{-1}(A))^\delta_{\ell}} \mathcal I (x) \ge \lim_{\varepsilon \to 0+} \inf_{y \in A^\varepsilon_{\ell_2} }  \tilde{\mathcal J}(y) = \inf_{y \in \cl A}  \tilde{\mathcal J}(y)  = \inf_{x \in F^{-1}(\cl A)} \mathcal I(x).
\]
By $\cl (F^{-1}(A))=F^{-1}(\cl A)$, the inequalities above are equalities, thus establishing~\eqref{eq: tight blow up}. The penultimate infimum is attained at some $y \in \cl A$ by tightness of $\tilde{\mathcal J}$, hence the last infimum is attained at some $x \in \cl B$, as needed.
\end{proof}

Our last simple claim, analogous to Lemma 4.1.5 in~\cite{DemboZeitouni}, describes the behaviour of metric LDPs under inclusions.

\begin{lemma} \label{lem: inclusions}
Let $(Z_n)_{n \ge 1}$ be a sequence of random elements of  a metric space $(\mathcal X, \ell)$ and  $\mathcal Y \subset \mathcal X$ be a Borel set such that $\P(Z_n \in \mathcal Y)=1$  for each $n \in \N$.

\begin{enumerate} %[label=\alph*), ref=\alph*, wide, labelwidth=!]
\item \label{item: restrict} If $(Z_n)_{n \ge 1}$ satisfies a metric LDP in $(\mathcal X, \ell)$ with some speed and a rate function $\mathcal I$ such that $\mathcal I = +\infty$ on $\mathcal Y^c$, then $(Z_n)_{n \ge 1}$ satisfies the metric LDP in $(\mathcal Y, \ell)$ with the same speed and the rate function $\mathcal I|_\mathcal Y$. 

\item \label{item: extend} Conversely, if $(Z_n)_{n \ge 1}$ satisfies a metric LDP in $(\mathcal Y, \ell)$ with some speed and a rate function $\mathcal I$, then $(Z_n)_{n \ge 1}$ satisfies the metric LDP in $(\mathcal X, \ell)$ with the same speed and the rate function $\cl \mathcal I$, where $\mathcal I$ is extended to $\mathcal X$ by putting $\mathcal I := + \infty$ on $\mathcal Y^c$. 
\end{enumerate}
\end{lemma}
\begin{proof}
\ref{item: restrict}) Clearly, $\mathcal I|_\mathcal Y$ is non-negative, lower semi-continuous, and not identically $+\infty$ (otherwise $\mathcal I = +\infty$ on $\mathcal X$). By monotonicity of probability, the lower bound in a metric LDP follows if we establish it for open sets. Since the topology of $(\mathcal Y, \ell)$ is the subspace topology induced from $(\mathcal X, \ell)$, for every set  $B \subset \mathcal Y$ that is open in $(\mathcal Y, \ell)$ we have $B= \tilde B \cap \mathcal Y$ for some $\tilde B \subset \mathcal X$ open in $(\mathcal X, \ell)$. Then, denoting by $(a_n)_{n \ge 1}$ the speed in the metric LDP,
\begin{equation} \label{eq: LDP inclusions}
- \inf_{y\in B} \mathcal I|_{\mathcal Y}(y) = - \inf_{x \in \tilde B} \mathcal I(x) \le  \liminf_{n\to\infty}\frac{1}{a_n}\log \P(Z_n \in \tilde B)  = \liminf_{n\to\infty}\frac{1}{a_n} \log \P(Z_n \in B), 
\end{equation}
where in the first equality we used that $\mathcal I = +\infty$ on $\mathcal Y^c$ and in the last one we used that $\P(Z_n \in \mathcal Y)=1$  for each $n$. This proves the lower bound required. The upper bound required follows from the fact that $\inf_{x \in B^\varepsilon_{\mathcal X}} \mathcal I(x) = \inf_{y\in  B^\varepsilon_{\mathcal Y}} \mathcal I|_{\mathcal Y}(y)$   for every $\varepsilon >0$.

\ref{item: extend}) For every $\tilde B \subset \mathcal X$ open in $(\mathcal X, \ell)$, the set $B:=\tilde B \cap \mathcal Y$ is open in $(\mathcal Y, \ell)$. Then \eqref{eq: LDP inclusions} holds true (switch the sides in both equalities), and the lower bound required follows from the equality $\inf_{x \in \tilde B} \mathcal I(x)  = \inf_{x \in \tilde B} \cl \mathcal I(x) $; see~\eqref{eq: regularization =}. And the upper bound follows from the fact that $\inf_{y\in  B^\varepsilon_{\mathcal Y}} \mathcal I(y) \ge \inf_{x \in B^\varepsilon_{\mathcal X}} \cl \mathcal I(x) $   for every $\varepsilon >0$, which holds true by $\mathcal I \ge \cl \mathcal I$.
\end{proof}

\section{The weak-* and related topologies on $BV[0,1]$} \label{sec: weak-*}

In this section we introduce the weak-* topology $W_*$ on the space of functions of bounded variation, then present a convenient metric topology $\widetilde{W}_*$ that coincides with $W_*$ on strongly bounded sets.
%, and conclude by comparing $\widetilde{W}_*$ with the Skorokhod %topologies and the Jakubowski topology $S$.

\subsection{The weak-* topology and a related metric} \label{sec: weak-* and rho_*}
Every $h \in BV[0,1]$ is the distribution function of the $\R^d$-valued finite Borel measure on $[0,1]$, which we denote by $dh$, that satisfies $dh([0,x])=h(x)$ for $x \in [0,1]$. As in the case $d=1$, this correspondence is bijective (Folland~\cite[Theorem~3.29]{Folland}). Note that~\cite{Folland} considers only complex-valued measures but all the cited results of \cite{Folland} are actually valid for any $d \ge 1$ since the consideration of $\R^d$-valued finite measures is coordinate-wise. For example, the integral  of a measurable function $f:[0,1] \to \R^d$ w.r.t.\ $dh$, is given by
\begin{equation} \label{eq: integral def}
\int_0^1 f \cdot dh := \sum_{k=1}^d \int_0^1 f^{(k)} d h^{(k)}, \qquad h \in BV[0,1],
\end{equation}
with the agreement that the notation above always means integration over $[0,1]$.

Recall that $\Var(h)$ denotes the  total variation of an $h \in BV[0,1]$; see~\eqref{eq: Var def}. This is a norm on $BV[0,1]$, and it generates a topology. Both will be referred to as  {\it strong}.
%, as opposed to the weak-* topology introduced below. 

%and it generates the {\it total variation} metric $\rho_{TV}(g,h):=\Var(g-h)$. 

Denote by $C[0,1]=C([0,1]; \R^d)$ the set of continuous functions on $[0,1]$, and equip it with the supremum norm  $\| \cdot \|_\infty$. By the Riesz theorem (\cite[Theorem~7.17]{Folland}), the dual of $(C[0,1], \| \cdot \|_\infty)$ is isometrically isomorphic to $(BV[0,1], \Var(\cdot))$ since we regard $BV[0,1]$ as the space of finite $\R^d$-valued Borel measures on $[0,1]$. In particular, we have
\begin{equation} \label{eq: var = operator}
\Var(h)= \sup_{f \in C[0,1]: \| f\|_\infty \le 1} \int_0^1 f \cdot dh, \qquad h \in BV[0,1],
\end{equation}
i.e.\ the strong (total variation) norm is the operator norm. The {\it weak-* topology} on $BV[0,1]$, denoted by $W_*$, is the coarsest topology such that all the linear functionals on $BV[0,1]$ of the form $h \mapsto \int_0^1 f \cdot dh$ for $f \in C[0,1]$ are continuous. The convergence defined by $W_*$ is called the weak-* convergence; it is traditionally referred to as weak convergence in probabilistic literature. 

%The key feature of $W_*$ is compactness of the strong norm balls 
%$$
%B_r:=\{h \in BV[0,1]: \Var(h) \le r \}, \qquad r >0,
%$$ 
%which holds by the Banach--Alaoglu theorem. 

We would want to apply our contraction principle (Theorem~\ref{thm: contraction general}) to the natural embedding  $(BV[0,1], \rho_1') \to (BV[0,1], W_*)$, but we should seek for a substitute of the weak-* topology $W_*$, which it is known to be non-metrizable. However, it is metrizable on strongly bounded subsets of $BV[0,1]$ (cf. a general metrization result~\cite[Theorem~V.5.1]{DunfordSchwartz}) 
%and~\cite[Lemma~2.9]{HumphreysSimpson}). 
%cf.\ a genereal Theorem 3.16 in Rudin on metrizability of weak-* convergence on weak-* compact subsets. 
This can be done using an explicit metric $\rho_*$ defined as follows.

Consider the norm
$$
{\|h\|}_* :=\int_0^1 |h(s)| ds + |h(1)|, \quad h \in BV[0,1].
$$
on $BV[0,1]$, which is simply the $L^1$-norm of $h$ w.r.t.\  the sum of the Lebesgue measure  on $[0,1]$ and the $\delta$-measure at $1$. The metric $\rho_*(g,h):={\|g-h\|}_*$ generates a topology, which we denote by $\widetilde{W}_*$. We have the following.
%Its relationship with the weak-* topology $W_*$ is as follows.

\begin{theorema}[H\"{o}gn\"{a}s~\cite{Hognas}] \label{thma: Hognas}
Suppose that $\{ g_\alpha\}_{\alpha \in A} \subset BV[0,1]$ is a strongly bounded net, i.e.\ $\sup_{\alpha \in A} Var(g_\alpha)  <\infty $. Then the following are equivalent: \\
1) $\lim_{\alpha \in A} \|g_\alpha\|_*=0$; \\
2) $\lim_{\alpha \in A} \int_0^1 f \cdot d g_\alpha = 0$ for any $f \in C[0,1]$, i.e.\ $\{ g_\alpha\}_{\alpha \in A}$ converges weakly-* to zero on $[0,1]$.
\end{theorema}

By~\eqref{eq: integral def}, this result fully reduces to $d=1$, the only case considered in~\cite{Hognas}. 

\begin{remark} \label{rem: unif bounded}
If the net $\{ g_\alpha\}_{\alpha \in A}$ is a sequence, i.e.\ $A=\N$, then by the uniform boundedness principle and \eqref{eq: var = operator},  $\sup_{\alpha \in A} \Var(g_\alpha) <\infty $  if and only if $\sup_{\alpha \in A} \big | \int_0^1 f \cdot d g_\alpha \big| <\infty$ for any $f \in C[0,1]$. Hence, a weakly-* convergent sequence also converges in the metric $\rho_*$ (but not vice versa).
\end{remark}

\begin{cor} \label{cor: Hognas}
A strongly bounded subset of $BV[0,1]$ that is closed or compact in one of the topologies $\widetilde{W}_*$ and $W_*$, is closed and compact in each of them.
\end{cor}

%Tao 245B, Notes 11: The strong and weak topologies Excercise 16 -- this uses that C[0,1] is a Banach space 

This result is the main reason why we have chosen to employ the weak-* topology.

\begin{proof}
It follows from the Banach--Alaoglu theorem that strongly bounded sets that are closed in $W_*$ are compact in $W_*$. Conversely, every set compact in the metric topology $\widetilde{W}_*$ is closed in $\widetilde{W}_*$. The claim then follows from  Theorem~\ref{thma: Hognas} because in a topological space, a set is closed if and only if together with any converging net it contains all its limits (Engelking~\cite[Corollary~1.6.4]{Engelking}), and it is compact if and only if any decreasing sequence of its closed non-empty subsets has non-empty intersection. 
\end{proof}

Let us clarify the relationship between the the topologies $\widetilde{W}_*$ and $W_*$. Denote by $\seq(W_*)$ the topology on $BV[0,1]$ where a set is closed if and only if it is {\it sequentially} closed in $W_*$; we will use this topology in Section~\ref{sec: integral}. We have $W_* \subset \seq(W_*)$ because in any topology, a closed set is sequentially closed. We also have $\widetilde{W}_* \subset \seq(W_*)$ since by Remark~\ref{rem: unif bounded}, sequential convergence in $W_*$ implies convergence in~$\widetilde{W}_*$.  However, this argument {\it does not} imply  that $\widetilde{W}_*$ is weaker than $W_*$ because it is known that  $W_* \neq \seq(W_*)$. And indeed, the topologies $\widetilde{W}_*$ and $ W_*$ are incomparable, with $\widetilde{W}_* \not \subset W_*$ following from the observation that for any $f_1, \ldots, f_n \in C[0,1]$, the set $\{h 
\in AC_0[0,1]: \int_0^1 f \cdot dh =0, 1 \le i \le n\}$ is unbounded in $\rho_*$.

\subsection{Comparison with the Skorokhod topologies} \label{sec: comparison}
First compare $\rho_*$ with the metrics $\rho_2$ and $\rho_2'$ defined in Section~\ref{sec: Skorokhod top}.

\begin{lemma} \label{lem: metrics}
For any $h \in BV([0,1]; \R^d)$ and $g \in D([0,1];\R^d)$, we have
$$
\int_0^1 |g(s)-h(s)|ds \le 2 d (\Var(h) - |h(0)| +1) \rho_2(g,h) + \pi d \rho_2^2(g,h)
$$
and
$$
\int_0^1 |g(s)-h(s)|ds \le 2 d (\Var(h)+1) \rho_2'(g,h) + \pi d (\rho_2')^2(g,h).
$$

\end{lemma}

\begin{proof}
We start by proving the first inequality for $d=1$. Consider the set
$$
U:=\big\{(s, x) \in \R^2: 0 \le s \le 1,  g(s) \wedge h(s) \le x \le  g(s) \vee h(s) \big \}.
$$
It is Borel because $g$ and $h$ are c\`adl\`ag on $[0,1]$. We claim that $U \subset \cl \big( (\Gamma h)^{\rho_2(g, h)} \bigr)$, where $(\Gamma h)^r $ denotes the Euclidean open $r$-neighbourhood  of $\Gamma h$, the completed graph of $h$. Then by Fubini's theorem, 
\begin{equation} \label{eq: content}
\int_0^1 |g(s)-h(s)|ds = \lambda(U) \le \lambda \big( \cl \big( (\Gamma h)^{\rho_2(g, h)} \bigr)\big).
\end{equation}
where $\lambda$ denotes the Lebesgue measure on the plane. 

In order to prove the claim, pick an $s \in [0,1]$. There is a point $(t, y) \in \Gamma h$ such that $|(t,y) - (s, g(s))| \le \rho_2(g,h)$. Hence a) $(s,x) \in \cl \big( (\Gamma h)^{\rho_2(g, h)} \bigr)$ for any $x \in  [g(s) \wedge y, g(s) \vee y]$, and b) since $\Gamma h = \gamma([0,1])$ for a continuous planar curve $\gamma \in \Pi h$ (see Section~\ref{sec: Skorokhod top}),  by the intermediate value theorem applied to the spatial coordinate of $\gamma$, for any $x \in [h(s) \wedge y, h(s) \vee y] $ there is a $u \in [s \wedge t, s \vee t]$ such that $(u,x) \in \Gamma h$, and by $|s-t| \le \rho_2(g, h)$ this implies $(s,x) \in \cl \big( (\Gamma h)^{\rho_2(g, h)} \bigr)$. Put together, a) and b) imply that $U \subset \cl \big( (\Gamma h)^{\rho_2(g, h)} \bigr)$, as claimed, by
$$
[g(s) \wedge h(s) , g(s) \vee h(s)] \subset [g(s) \wedge y, g(s) \vee y] \cup [h(s) \wedge y, h(s) \vee y].
$$

Furthermore, denote by $\ell (\gamma)$ the length of $\gamma$. It is easy to check, using the definition of the total variation of $h$, that 
\begin{equation} \label{eq: length <=}
\ell (\gamma) \le \Var(h)-|h(0)|+1. 
\end{equation}
Thus, $\gamma$ is a rectifiable curve (i.e.\ offinite length), therefore (Federer~\cite[Theorem~3.2.39]{Federer})
\begin{equation} \label{eq: Minkowski}
\ell (\gamma)= \lim_{r \to 0+}\lambda(\cl ((\Gamma h)^r)) /(2 r),
\end{equation}
where the limit is known as the one-dimensional {\it Minkowski content} of the set $\Gamma h$. 

On the other hand, for any compact connected planar set $F$, the function $r \mapsto \lambda(F^r) - \pi r^2$ is known to be concave on $(0, \infty)$ (Fast~\cite[Theorem on p.~139]{Fast} and Sz\H{o}kefalvi-Nagy~\cite[Theorem 1]{Sz.-Nagy}). Its right derivative at $0$ is $2 \ell(\gamma)$ by \eqref{eq: Minkowski}. Then, since $\Gamma h$ is connected, 
\begin{equation} \label{eq: Steiner}
\lambda \big( \cl \big( (\Gamma h)^r \bigr)\big) \le 2 r \ell(\gamma) + \pi r^2, \qquad r >0.
\end{equation}
This inequality, which  is sometimes referred to as {\it Steiner's inequality} (cf.~Steiner's formula), is actually available in~\cite[p.~146]{Fast}; the assumptions imposed in \cite{Fast} are satisfied since $\gamma$ is a rectifiable simple (i.e.\ injective) curve.  

Put together, inequalities \eqref{eq: content}, \eqref{eq: length <=}, \eqref{eq: Steiner} imply the first inequality of Lemma~\ref{lem: metrics} for $d=1$. This in turn proves the inequality in any dimension using that  $|x| \le |x^{(1)}| + \ldots + |x^{(d)}|$ for $x \in \R^d$, and $\Var(h^{(k)}) - |h^{(k)} (0)| \le \Var(h) - |h(0)|$ and $\rho_2(g^{(k)}, h^{(k)}) \le \rho_2 (g, h)$ for $k=1, \ldots, d$. The last inequality can be obtained from  definition \eqref{eq: rho Hausd} of the metric~$\rho_2$ as follows: first estimate $|(s,x)-(t,y)| \ge |(s,x^{(k)})-(t,y^{(k)})| $ and then, since the r.h.s.\ of this inequality does not depend on the remaining coordinates,  eliminate them from the constraints under the maximum and the minimum in~\eqref{eq: rho Hausd}.

The second inequality of Lemma~\ref{lem: metrics} follows by the same argument using that by \eqref{eq: rho and rho'}, the modified completed graphs $\Gamma' g$ and $\Gamma' h$ can be regarded as the usual completed graphs of the modified functions $g_0:= g \I_{(0,1]}$ and $h_0:= h \I_{(0,1]}$, which are c\`adl\`ag on $(0,1]$ and have right limits at $0$. Hence $\rho_2'(g,h)=\rho_2(g_0, h_0)$, and if $\gamma \in \Pi'(h)$ is a parametrization of $\Gamma' h$, then $\ell(\gamma) \le \Var(h)+1$. It remains to use that $\int_0^1 |g(s)-h(s)|ds =  \int_0^1 |g_0(s)-h_0(s)|ds$.
\end{proof} 

We now use Lemma~\ref{lem: metrics} to clarify the relationship between the  topologies introduced. With no risk of confusion, in the rest of the paper we use the original notation $M_i$, $M_i'$
for the induced topologies on  $BV[0,1]$. 
%Recall that the $M_i$ is generated by the metric~$\rho_i$, $i =1,2$, where $\rho_2' \le \rho_1'$. 
We have $\widetilde{W}_*  \subset M_1'$
%\begin{equation} \label{eq: topologies *}
%\widetilde{W}_*  \subset M_1 
%\end{equation}
since $\rho_2' \le \rho_1'$ and convergence of c\`adl\`ag functions in the metric $\rho_1'$ implies convergence of their values at the endpoint $1$. 

Neither $\widetilde{W}_* $ nor $W_*$ is comparable with $M_2$. For example, for $g_n:=\I_{[1-1/n,1)}$ and $g:=\I_{\{1\}}$, we have $\rho_2(g_n,g) \to 0$ but $\rho_*(g_n,g) \not \to 0$ as $n \to \infty$. However, by Lemma~\ref{lem: metrics}, convergence in either $\rho_2$ or $\rho_2'$  implies convergence in $\rho_*$ if the limit function is continuous at~$1$.
Moreover, $W_*$ is incomparable with $M_1$. For example, for $g_n:=\{n \, \cdot \}/\sqrt{n}$, where $\{ \cdot \}$ denotes the fractional part, we have $\rho_1(g_n,0) \to 0$ but $g_n$ does not converge weakly-* since its total variation explodes.  Likewise, $W_*$ is incomparable with the topology of uniform convergence, which is weaker than the strong topology on $BV[0,1]$ by inequality \eqref{eq: supnorm <= var} below. 
%\begin{equation*}
%\rho(g,h) \le \|g-h\|_\infty \le \Var(g-h), \qquad g,h \in BV[0,1],
%\end{equation*}
%where the second equality holds by \eqref{eq: supnorm <= var} below. 

\section{Application of the weak-*-related topologies to the study of $I_D$} \label{sec: integral}

In this section we use the topologies $\widetilde{W}_*$ and $\seq(W_*)$ on $BV[0,1]$, introduced in Section~\ref{sec: weak-*}, to study properties of the rate function $I_D$ using the results of variational calculus. Namely, we prove {\it sequential weak-* lower semi-continuity} of $I_D$ and use this property to obtain an explicit {\it integral representation} for $I_D$, which is written using the {\it directional decomposition} of the total variation of functions in $BV([0,1]; \R^d)$. Moreover, we standardize the upper bound in the metric LDP of Borovkov and Mogulskii~\cite{BorovkovMogulskii2, BorovkovMogulskii3} for a few types of sets, and prove a standard LDP for trajectories of random walks  in the space $(BV[0,1], \rho_*)$. 

\subsection{Directional decomposition of total variation} Recall that $AC_0[0,1]$ denotes the set of coordinate-wise absolutely continuous functions  from $[0,1]$ to $\R^d$ that equal $0$ at $0$. These are exactly the distribution functions of $\R^d$-valued finite Borel absolutely continuous measures on $[0,1]$. For any $h \in BV[0,1]$, put $h_a(t):= \int_0^t h'(s) ds$, where  $h'$ exists a.e.\ and is integrable by \cite[Proposition~3.30]{Folland}, which also ensures that the measure $dh_s:=dh-dh_a$ is singular. We say that $h=h_a +h_s$ is the {\it Lebesgue  decomposition} of the vector-valued function $h$;  let us stress  that $h_a \in AC_0[0,1]$.

Denote by $V^h$ the total variation {\it function} of an $h \in BV[0,1]$, defined by $V^h(t):=\Var(h(\cdot \wedge t))$ for $t \in [0,1]$; cf.~\eqref{eq: Var def}. It is non-decreasing,  c\`adl\`ag, and satisfies $\Var(h)=V^h(1)$, hence $V^h \in BV[0,1]$. By \cite[Theorem~3.29 and Exercise~21 in Section~3.3]{Folland}, $d V^h$ is the total variation {\it measure} of the vector-valued measure $d h$, that is the equality
\begin{equation} \label{eq: TV measure}
dV^h(B) = \sup \left \{\sum_{i=1}^\infty |dh(B_i)|: B_1, B_2, \ldots \text{ are disjoint Borel sets}, \, \bigcup_{i=1}^\infty B_i = B \right \}
\end{equation}
holds true for every Borel set $B\subset[0,1]$. This implies that $V^{h_1+h_2}=V^{h_1} + V^{h_2}$ whenever $d h_1$ and $d h_2$ are singular. In particular,  we have $V^h=V^{h_a} + V^{h_s}$.
% by using that $d V^h(B)=d V^h(N \cap B) + d V^h(N^c \cap B)$ with any %Borel set $N $ such that $dh_a(N)=dh_s(N^c)=0$.

%In particular, this yields
%\begin{equation} \label{eq: metrics2}
%\rho(g,h) \le \| g-h\|_\infty \le \Var (g-h), \qquad g, h \in BV[0,1].
%\end{equation}
%hence the topology of total on $BV[0,1]$ is stronger than the Skorokhod topology $J_1$. 
It holds $d h \ll d V^h$ and the Radon--Nykodim density $\dot{h}:[0,1] \to \R^d$, defined by $d h = \dot h  \, d V^h $, satisfies $|\dot h|=1$ $d V^h$-a.e.\ (\cite[Proposition~3.13.b]{Folland}). In particular, this implies that
\begin{equation} \label{eq: supnorm <= var}
\|h\|_\infty \le \Var(h), \qquad h \in BV[0,1].
\end{equation}

We say that the push-forward measure $d \sigma^h:= d V^h \circ (\dot h)^{-1}$ on the unit sphere $\S^{d-1}$ is the {\it directional decomposition} of the total variation of $h$. For example, if $d=1$, the Hahn--Jordan decomposition gives the unique representation $h=h^+-h^-$, where $h^\pm \in BV[0,1]$ are non-decreasing functions, and so $d \sigma^h= h^+(1) \delta_1 + h^-(1) \delta_{-1}$. Then $\Var(h) =\sigma^{h}(\S^{d-1})$ and 
\begin{equation} \label{eq: TV}
\Var(h) = \int_0^1 |h'(t)|dt + \sigma^{h_s}(\S^{d-1}).
\end{equation}

To prove \eqref{eq: TV}, note that it follows from \eqref{eq: TV measure} that the measure $dV^{h_a}$ is absolutely continuous (because so is $d h_a$). Then from the equalities  $h' dt = d h_a = \dot{h_a} \, d V^{h_a}$ and $|\dot{h_a}|=1$ $d V^{h_a}$-a.e., we see by equating densities that $d V^{h_a} = |h'| dt$ (because the unit vector $\dot{h_a}$ multiplied by the scalar density of $d V^{h_a}$ equals $h'$). This implies \eqref{eq: TV} by $d V^h = d V^{h_a} + d V^{h_s}$.

\subsection{Lower semi-continuity of $I_D$ w.r.t.\ $\rho_*$ and related results} \label{sec: variational}
%For such topologies, which of course include the metrizable ones, the %properties continuity and lower semi-continuity of a function are sequential, %e.g.\ continuity and sequential continuity in $S$ are equivalent. % Proposition %1.1.5(ii) in Buttazzo and 
%But an extra care is need since compactness is not a sequential property! 

%Recall that $\seq(W_*)$ denotes the topology on $BV[0,1]$ where a set is %closed if and only if it is sequentially weak-* closed and $\widetilde{W}_*$ %is the topology generated by the metric~$\rho_*$. and the Jakubowski %topology $S$ was introduced in Section~\ref{sec: Jakubowski}. 

%We will work with closures of sets in the topologies $\widetilde{W}_*$ and %$\seq(W_*)$. The difference between the two is as follows. From Theorem~%\ref{thma: Hognas} and Remark~\ref{rem: unif bounded}, we see that the %closure in $\seq(W_*)$ consists of all possible $\rho_*$-limits of strongly %bounded subsequences contained in the set, while the for the closure in $%widetilde{W}_*$ we take the $\rho_*$-limits over all subsequences. 

Recall that $\mathcal D_{\mathcal L}$ is the subset of $\R^d$ where the Laplace transform of $X_1$ is finite. This set is convex. Denote by
\begin{equation} \label{eq: support function}
I_\infty(v):=\sup\{u \cdot v: u \in \mathcal{D}_{\mathcal L}\}, \qquad v \in \R^d,
\end{equation}
its {\it support  function}. This name reflects that $I_\infty$ equals the so-called {\it recession function} of~$I$ (Rockafellar~\cite[Theorem~13.3]{Rockafellar}), which is convex, lower semi-continuous and positively homogeneous on $\R^d$, and has the property (\cite[Theorem~8.5]{Rockafellar}) 
\begin{equation} \label{eq: recession function}
I_\infty(v)=\lim_{t \to \infty} I(u + v t)/t = \sup_{t >0} \big [(I(u+vt) - I(u) )/t \big ], \qquad u \in \mathcal{D}_{\mathcal L}, v \in \R^d.
\end{equation}
Note that in dimension $d=1$, we have $\intr \mathcal D_{\mathcal L} = (-I_\infty(-1), I_\infty(1))$.

We can now state the main result of the section.

\begin{theorem} \label{thm: I_D}
Assume that $X_1$ is a random vector in $\R^d$ such that $0 \in \intr \mathcal D_{\mathcal L}$. Then the functional $I_D$ on $BV[0,1]$, defined in \eqref{eq: I_D def}, is tight w.r.t.\ $\rho_*$. Moreover, we have
\begin{equation} \label{eq: I_D = integral}
I_D(h)= \int_0^1 I(h'(t)) dt + \int_{\S^{d-1}} I_\infty(\ell) \, \sigma^{h_s}(d \ell), \qquad h \in BV[0,1],
\end{equation}
which is dimension $d=1$ reads as
\begin{equation} \label{eq: I_D 1-D}
I_D(h)= \int_0^1 I(h'(t)) dt + h_s^+(1) I_\infty(1) + h_s^-(1) I_\infty(-1).
\end{equation}
\end{theorem}
Formula \eqref{eq: I_D 1-D} is available in~\cite[Theorem~3.3]{BorovkovMogulskii2}. If $\mathcal D_{\mathcal L} = \R^d$, that is the Laplace transform of $X_1$ is finite on $\R^d$, then $I_\infty(v)= +\infty$ for $v \neq 0$, hence $I_D(h) = +\infty$ for $h \not \in AC_0[0,1]$. 

The advantage of integral representation \eqref{eq: I_D = integral} is in its explicitness. It becomes more transparent when compared with equality~\eqref{eq: TV}, where the total variation of the singular component of a function is expressed using its directional decomposition. We can get \eqref{eq: TV} by formally substituting the Euclidean norm $|\cdot|$ for $I$ in \eqref{eq: I_D = integral}.
%; of course the Euclidean norm $|\cdot|$ is not a rate function.   

%  and the ease to work with it using the results of variational calculus. 

The next two statement are corollaries to Theorem~\ref{thm: I_D}.

\begin{prop} \label{prop: LDP weak}
Assume that $X_1$ is a random vector in $\R^d$ such that $0 \in \intr \mathcal D_{\mathcal L}$. Then 
both random sequences $(S_n(\cdot)/n)_{n \ge 1}$ and $(S_{[n \cdot]}/n)_{n \ge 1}$  satisfy the (standard) LDP  in the separable metric space $(BV[0,1], \rho_*)$ %and in the separable Hausdorff space $(BV[0,1], S)$ 
with the {\it tight} convex rate function $I_D$. 
\end{prop}

%The topology $\widetilde{W}_*$, generated by the metric $\rho_*$, is weaker than $M_1$. Therefore %Proposition~\ref{prop: LDP weak} gives worse bounds for large deviations probabilities than Theorem~%\ref{thma: BM}. However, there is  one important case when the bounds match -- 

The main application of this result is our LDP for  kernel-weighted sums  of i.i.d.\ random vectors in $\R^d$, presented in Section~\ref{Sec: weighted}. It is worth to compare  Proposition~\ref{prop: LDP weak} with the results by Gantert~\cite[Theorems~1 and~2]{Gantert}, who proved LDPs in $L^1$ (i.e., ``almost'' in the same topology as in our case but on a different space) for one-dimensional random walks with so-called semi-exponential increments.

\begin{proof}[{\bf Proof of Proposition~\ref{prop: LDP weak}.}]
This follows from Theorem~\ref{thma: BM} combined with Theorem~\ref{thm: contraction general} applied to the natural embedding $F: (BV[0,1], \rho_1') \to (BV[0,1], \rho_*)$, where $\tilde{\mathcal J} = \mathcal I = I_D$, and $I_D$ is a tight rate function w.r.t.\ $\rho_*$ by Theorem~\ref{thm: I_D}. The assumptions of Theorem~\ref{thm: contraction general} are satisfied because by Lemma~\ref{lem: metrics} and the inequality  $\rho_2' \le \rho_1'$, $F$ is continuous and it is uniformly continuous on strongly bounded subsets of $BV[0,1]$, and also on the sub-level sets of $\mathcal I$ by~\eqref{eq: I_D bound}. Lastly, by Corollary~\ref{cor: Hognas}, the metric space $(\{h: \Var(h) \le n \}, \rho_*)$ is compact, hence totally bounded, hence separable for every $n \in \N$, therefore $ (BV[0,1], \rho_*)$ is separable. 
\end{proof}

\begin{remark}
Note that we cannot apply Theorem~\ref{thm: contraction} instead of Theorem~\ref{thm: contraction general} in the proof presented because the metric space $(BV[0,1], \rho_*)$ is not complete. A way around is to consider the natural embedding of $(BV[0,1], \rho_1)$ into the complete metric space $(L^1[0,1], \rho_*)$.

Namely, for any $g \in L^1[0,1]$, put $I_D(g):=I_D(h)$ if there exists an $h \in BV[0,1]$ such that $\rho_*(g,h)=0$ and $I_D(g):=+\infty$ otherwise. Let us check that this extended version of $I_D$ remains lower semi-continuous. Since $I_D$ has this property on $(BV[0,1], \rho_*)$ by Theorem~\ref{thm: I_D}, it suffices to prove that for any $h_1, h_2, \ldots \in BV[0,1]$ and $g \in L^1[0,1] \setminus BV[0,1]$ such that $\lim_{n \to \infty} \rho_*(h_n, g)=0$ and $\rho_*(g,h)>0$ for every $h \in BV[0,1]$, we have $\liminf_{n\to\infty} I_D(h_n)=+\infty$. This follows from \eqref{eq: I_D bound}  because $\liminf_{n\to\infty} \Var (h_n)=+\infty$ (no subsequence of $(h_n)_n$ is  strongly bounded since otherwise by Corollary~\ref{cor: Hognas} we can choose a further subsubsequence converging in $\rho_*$ to some $h \in BV[0,1]$, hence we arrive at the contradictory $\rho_*(g,h)=0$). 

Theorem~\ref{thm: contraction} then implies that  $(S_n(\cdot)/n)_{n \ge 1}$ and $(S_{[n \cdot]}/n)_{n \ge 1}$ satisfy the LDP in $(L^1[0,1], \rho_*)$ with the tight rate function $I_D$; it reduces to the LDP in $(BV[0,1], \rho_*)$ by \cite[Lemma 4.1.5.b]{DemboZeitouni}.
\end{remark}

Let us use subscripts to indicate in which metric (or topology) we take closures.

\begin{prop} \label{prop: upper}
Assume that $X_1$ is a random vector in $\R^d$ such that $0 \in \intr \mathcal D_{\mathcal L}$. Let $B \subset BV[0,1]$ be such that $\cl_{\rho_1'}(B) = \cl_{\rho_*}(B)$ (where $\rho_1'$ is restricted to $BV[0,1]$). Then 
$$
 \inf_{h \in \cl_{\rho_1'}(B)} I_D(h) = \lim_{\varepsilon \to 0+}\inf_{h \in B^\varepsilon_{\rho_1'}}I_D(h),
$$
with the infimum on the l.h.s.\ attained at some $h \in \cl_{\rho_1'}(B)$ unless $B$ is empty.
\end{prop}

We prove this applying Proposition~\ref{prop: upper general} to the natural embedding as in the proof of Proposition~\ref{prop: LDP weak}. We thus see that for the sets $B$ that are closed both in $\rho_1'$ and $\rho_*$, the upper bound  in the metric LDP of Theorem~\ref{thma: BM} matches the standard LDP one.

\begin{remark} \label{rem: closures}

A set $B$ is closed in $\rho_*$ and in $\rho_1'$ when it is a sub-level set of a functional on $BV[0,1]$ that is lower semi-continuous in both metrics. Examples of such functionals include:

\begin{enumerate}
\item The {\it action} functional $J_D$ defined by the r.h.s.\ of \eqref{eq: I_D = integral} with $I$ replaced by any {\it convex} lower semi-continuous  function $J: \R^d \to [0, +\infty]$ that satisfies $J(u) \ge c_1 |u| - c_2$ for some $c_1, c_2 >0$ and every $u \in \R^d$; for example, from \eqref{eq: TV} we see that $J_{|\cdot|}=\Var$. The action functional $J_D$ is sequentially weak-* lower semi-continuous by~\cite[Corollary 3.4.2]{Buttazzo} applied exactly as in the proof of Theorem~\ref{thm: I_D} below. By Theorem~\ref{thma: Hognas} and Remark~\ref{rem: unif bounded}, $J_D$ is also lower semi-continuous in $\rho_*$ (and hence in the stronger metric~$\rho_1'$) because its sub-level sets are strongly bounded. Indeed, from \eqref{eq: TV}  it follows that $J_D(h) \ge c_1 \Var(h) - c_2$ for $h \in BV[0,1]$ since $J_\infty(u) \ge c_1$ for $u \in \R^d$ by \eqref{eq: recession function}.

\item \label{item: maximum} The {\it maximum} functional $h \mapsto \sup_{0 \le t \le 1} (h(t) \cdot \ell)$, where $\ell \in \S^{d-1}$ is a fixed direction. To check its lower semi-continuity in $\rho_*$ (which suffices since $\widetilde{W}_* \subset M_1'$), note that the value of the functional on an $h \in BV[0,1]$ is either $h(t_0-) \cdot \ell$ or $h(t_0) \cdot \ell$ for some $t_0 \in [0,1]$. This fact, combined with the {\it c\`adl\`ag}  property of $h$ and the fact that  every $\rho_*$-convergent sequence contains a subsequence that converges pointwise on a dense subset of $[0,1]$  that contains $1$, yields the required property of the functional.

Note in passing that the maximum functional is continuous in $\rho_i$ but not in $\rho_i'$. However, the {\it positive maximum} $h \mapsto  \sup_{0 \le t \le 1} (h(t) \cdot \ell)_+$ is continuous in $\rho_i'$.

\item For $d=2$, the {\it perimeter of the convex hull} (and the {\it mean width} in higher dimensions) of  $h([0,1])$, the image of a planar curve $h$. Large deviations of the perimeter of the convex hull of the first $n $ steps of a planar random walk were studied by Akopyan and Vysotsky~\cite{AkopyanVysotsky}. For the perimeter functional, lower semi-continuity in $\rho_*$ follows from the combination of Cauchy's formula for  perimeter of a planar convex set, 
% (\cite[Eq.~(53)]{AkopyanVysotsky}), 
the result of Example~\ref{item: maximum}), and Fatou's lemma. 

The perimeter functional is continuous in $\rho_i$ but not in $\rho_i'$.  However, the perimeter of the convex hull of the set $h([0,1]) \cup \{0\}$ is continuous in $\rho_i'$.

One more type of sets  satisfying the assumption of Proposition~\ref{prop: upper} is as follows.

\item $B_f:=\{h \in BV[0,1]: h \le f\}$ for some $f \in C[0,1]$. This set is closed in $\rho_*$ (and $\rho_1'$) by the same argument as we used in Example~\ref{item: maximum}). If $B_f \cap \mathcal D_{I_D} \neq \varnothing$, the minimizers  of $I_D$ over $B$ exist, and we call them the {\it taut strings}. In a probabilistic setup, taut strings were considered by Lifshits and Setterqvist~\cite{LifshitsSetterqvist}, who were interested in those corresponding to the sets of the form  $\{h \in BV[0,1]: f_1 \le h \le f_2\} = B_{f_1} \cap (-B_{-f_2})$ for $f_1, f_2 \in C[0,1]$.
\end{enumerate} 
\end{remark}

\subsection{Proof of Theorem~\ref{thm: I_D}}
We will use results and methods of the calculus of variations, referring to the book by Buttazzo~\cite{Buttazzo}. The action (integral) functionals on the spaces of finite Borel vector-valued measures are considered in Chapter~3 of~\cite{Buttazzo}, where the notation $C_0([0,1]; \R^d)$ corresponds to our $C[0,1]$; see~\cite[Section~3.1]{Buttazzo}.
Consider an integral functional $I_C$ defined by $I_C(h):=\int_0^1 I(h')dt$ for $h \in AC_0[0,1]$, and extend it formally to $BV[0,1]$ by putting $I_C(h):=+\infty$ for $h \not \in AC_0[0,1]$. The main  idea, which applies in a more general setup as described in Section~1.3 of~\cite{Buttazzo}, is that a natural extension is actually given by $\cl_{\seq(W_*)} (I_C)$, referred to in \cite{Buttazzo} as {\it relaxed functional}. The book offers results which will allow us to find this extension explicitly, and we will show that it equals $I_D$.

%, i.e.\ the function whose epigraph is the closure in the topological space $%(BV[0,1], \seq(W_*)) \times ([0, +\infty], | \cdot|)$ of the epigraph of~%$I_C$.

%this book conserns, if explained in our setting, the space of finite Borel %vector-valued signed measures on an {\it open} interval $(a, b) \subset \R$, %which is dual to the space $(C_{0,0}[a,b], \| \cdot\|_\infty)$ of continuous %functions $f$ on $[a,b]$ such that $f(a)=f(b)=0$. All the results of~%\cite{Buttazzo} apply to our case if we extend our functions by continuity  from %the closed interval $[0,1]$ to, say,  $(-0.1, 1.1)$ by setting them to be %constant outside $[0,1]$. The details of such extension are trivial and %therefore omitted. 

%\begin{proof}[{\bf Proof of Theorem~\ref{thm: I_D}}] 
We have $I_C=I_D$ on $AC_0[0,1]$ by~\cite[Theorem~5.3]{BorovkovMogulskii2}, hence
\begin{equation} \label{eq: I_D <= I_C}
I_D(h) \le I_C(h), \qquad h \in BV[0,1].
\end{equation}
In the new notation, the definition~\eqref{eq: I_D def} of $I_D$ reads as 
\begin{equation} \label{eq: I_D = new}
I_D(h) = \sup_{\mathbf t \subset (0,1]: \, \# \mathbf t < \infty } I_C(h^{\mathbf t}), \qquad h \in BV[0,1].
\end{equation}
Then for any dense sequence $(t_n)_{n \ge 1}$ in $(0,1]$, for $\mathbf{t}_n:=\{t_1, \ldots, t_n\}$ we have 
\begin{equation} \label{eq: I_D = limit}
I_D(h)= \lim_{n \to \infty}  I_C(h^{\mathbf{t}_n}), \qquad h \in BV[0,1].
\end{equation}
because $I_D$ is lower semi-continuous w.r.t.\ $\rho_2'$ by Theorem~\ref{thma: BM} and  $\rho_2'(h^{\mathbf{t}_n}, h) \to 0$, which follows from the c\`adl\`ag property of $h$.  Moreover, $h^{\mathbf{t}_n} \to h$ weakly-*. Indeed, Lemma~\ref{lem: metrics} and the equality  $h^{\mathbf{t}_n}(1) = h(1)$ imply that $\rho_*(h^{\mathbf{t}_n} , h) \to 0$, and the weak-* convergence then follows from Theorem~\ref{thma: Hognas}, which applies since $\Var(h^{\mathbf{t}_n}) \le \Var(h)$ by \eqref{eq: Var def}. 

Let us prove lower semi-continuity of $I_D$  in the metric $\rho_*$. Use that lower semi-continuity in metric spaces is a sequential property. Assume  that there are $g, g_1, g_2, \ldots \in BV[0,1]$ such that $\rho_*(g_n, g) \to 0$ but $I_D(g) > \liminf_n I_D(g_n)$ as $n \to \infty$. Since $\rho_*(g_n, g) \to 0$ means convergence in $L_1$ and $g_n(1) \to g(1)$, by considering a subsequence, we can assume w.l.o.g.\ that the convergence is point-wise on a subset of $(0,1]$ of full Lebesgue measure.
Pick a sequence $(s_n)_{n \ge 1}$ of distinct elements of  this set that is dense in $(0,1]$ and satisfies $s_1=1$.

%Note that the use of $\rho*$ is not really needed here since any weakly-* convergent sequence converges point-wisely at the points of continuity of the limit by Proposition 2.27 in Buttazzo et al..

For any integer $k \ge 1$, put  $\mathbf{s}_k:=\{s_1, \ldots, s_k\}$, and let $\sigma_k$ be the permutation of length $k$ such that $s_{\sigma_k(1)}< \ldots < s_{\sigma_k(k)}$. For any $i, k, n \in \N$ satisfying $1 \le i  \le k$, we have  $g_n^{\mathbf{s}_k}(s_i)=g_n(s_i)$ and also $g_n(s_{\sigma_k(i)}) \to g(s_{\sigma_k(i)})$ as $n \to \infty$. Hence, by lower semi-continuity and non-negativity of $I$, for any fixed integer $k \ge 1$ we have
\begin{align*}
I_C(g^{\mathbf{s}_k}) &=s_{\sigma_k(1)} I\Big( \frac{g(s_{\sigma_k(1)})}{s_{\sigma_k(1)}} \Big) + \sum_{i=1}^{k -1} (s_{\sigma_k(i+1)}-s_{\sigma_k(i)}) I\Big( \frac{g(s_{\sigma_k(i+1)})-g(s_{\sigma_k(i)})}{s_{\sigma_k(i+1)}-s_{\sigma_k(i)}}\Big) \\
&\le \liminf_{n \to \infty} \bigg [ s_{\sigma_k(1)} I\Big( \frac{g_n(s_{\sigma_k(1)})}{s_{\sigma_k(1)}} \Big)  + \sum_{i=1}^{k -1} (s_{\sigma_k(i+1)}-s_{\sigma_k(i)}) I\Big( \frac{g_n(s_{\sigma_k(i+1)})-g_n(s_{\sigma_k(i)})}{s_{\sigma_k(i+1)}-s_{\sigma_k(i)}}\Big) \bigg ] \\
&= \liminf_{n \to \infty} I_C(g_n^{\mathbf{s}_k}).
\end{align*}
From \eqref{eq: I_D = new} we see that $I_C(g_n^{\mathbf{s}_k}) \le  I_D(g_n) $, hence $I_C(g^{\mathbf{s}_k}) \le \liminf_{n \to \infty} I_D(g_n)$. It remains to take $k \to \infty$ and use \eqref{eq: I_D = limit}  to arrive at $I_D(g) \le \liminf_{n \to \infty}  I_D(g_n)$, which contradicts our assumption that the lower semi-continuity does not hold.

Furthermore, the sub-level sets of $I_D$ are strongly bounded by \eqref{eq: I_D bound}. They are closed in the metric $\rho_*$ since $I_D$ is lower semi-continuous in $\rho_*$. Therefore, they are compact by Corollary~\ref{cor: Hognas}, and thus  $I_D$ is tight, as claimed. 

It remains to prove integral representation \eqref{eq: I_D = integral}. Denote by $\mathcal I(h)$ its r.h.s. By~\cite[Corollary 3.4.2]{Buttazzo} (which applies  because the change of measure formula in the definition of directional total variation $\sigma^h$ brings to $\mathcal I$ into the form~\cite[Eq.~(3.4.1)]{Buttazzo}  and Condition (i) in \cite[Lemma~2.2.3]{Buttazzo} is satisfied for $z_0=\E X_1$), the functional $\mathcal I$ is sequentially weakly-* lower semi-continuous on $BV[0,1]$. Equivalently, $\mathcal I$ is lower semi-continuous in $\seq(W_*)$, the topology  where a set is closed if and only if it is sequentially weakly-* closed (\cite[Proposition~1.1.5(ii)]{Buttazzo}). 

We claim that  $\mathcal I = \cl_{\seq(W_*)} (I_C)$. It holds $\mathcal I \le \cl_{\seq(W_*)} (I_C) $ because the on r.h.s.\ we have the maximal functional that is lower semi-continuous in $\seq(W_*)$ and dominated by $I_C$ (\cite[Propositions~1.1.2(ii)]{Buttazzo}). Therefore, since $\mathcal I = I_C$ on $AC_0[0,1]$, we have $\cl_{\seq(W_*)} (I_C) = I_C$ on $AC_0[0,1]$.
On the other hand, by \cite[Theorem~3.3.1]{Buttazzo}, 
$\cl_{\seq(W_*)} (I_C)$  on $AC_0[0,1]$ is of the form $h \mapsto \int_0^1 J(t, h'(t)) dt$ for some  measurable function $J:[0,1] \times \R^d \to [0, +\infty]$ such that $J(t, \cdot)$ is convex and lower semi-continuous for a.e.\ $t$. This is possible only when $J(t, \cdot) =I$ for a.e.~$t$ by  \cite[Proposition~2.1.3]{Buttazzo} (applied with $\psi=0$), hence $\mathcal I = \cl_{\seq(W_*)} (I_C)$, as claimed.

Furthermore, $I_D$ is lower semi-continuous in $\seq(W_*)$  because $I_D$ is lower semi-continuous in the metric $\rho_*$ and the topology  $\widetilde{W}_*$ generated by $\rho_*$ is coarser than $\seq(W_*)$ (see Section~\ref{sec: weak-* and rho_*}). On the other hand, by  lower semi-continuity of $\mathcal I$ in $\seq(W_*)$ and equality \eqref{eq: I_D = limit}, where $I_C(h^{\mathbf{t}_n})= \mathcal I (h^{\mathbf{t}_n})$ and $h^{\mathbf{t}_n} \to  h$ weakly-* as $n \to \infty$ (this is explained right after~\eqref{eq: I_D = limit}), we have  $\mathcal I \le I_D$. We also have $I_D \le I_C$ by \eqref{eq: I_D <= I_C}. Therefore, since $\mathcal I = \cl_{\seq(W_*)} (I_C) $ and $\cl_{\seq(W_*)} (I_C) $ is the maximal functional that is lower semi-continuous in $\seq(W_*)$ and dominated by $I_C$, we have $\mathcal I = I_D$, as required.

\section{Cram\'er's theorem for kernel-weighted sums} \label{Sec: weighted}
In this section we present an application of Proposition~\ref{prop: LDP weak}, yielding an LDP for {\it kernel-weighted sums} of i.i.d.\ random vectors in $\R^d$. This extends the classical Cram\'er theorem. 

It appears that large deviations of weighted i.i.d.\ random variables were first studied by Brook~\cite{Brook}. The next results were due to Kiesel and Stadtm{\"u}ller~\cite{KieselStadtmuller2000}, who considered only the ``light-tailed'' case where the i.i.d.\ terms have finite Laplace transform, i.e.\ $\mathcal D_{\mathcal L}=\R$. They proved an LDP and found the rate function (available in \cite[Theorem on p.~933]{KieselStadtmuller2000} with the coefficients $a_\nu$ for the kernel weights given in~\cite[p.~976]{KieselStadtmuller1996}); below we will present a more explicit expression~\eqref{eq: weight rate f 1-D} for the rate function. A  recent result~\cite[Theorem~3]{Gantert+} by Gantert et al.\ is an LDP (with  polynomial speed and an explicit rate function) for kernel-weighted sums of i.i.d.\ random variables with stretched (super-) exponential tails. In this ``heavy-tailed'' case where the random variables have no exponential moments, the rate function is defined by the supremum of the (non-negative) kernel. This corresponds to the first  two terms in our formula~\eqref{eq: weight rate f 1-D} (cf.~Remark~\ref{rem: weighted}.\ref{item: minimizer}). Since these terms vanish for ``light-tailed'' increments, in the case $0 \in \intr\mathcal D_{\mathcal L} $ with $\mathcal D_{\mathcal L} \neq \R$ there is a natural transition of the rate function from the ``light-tailed'' case to the ``heavy-tailed'' one. 

We introduce more notation to state our result. Let us agree to write $\sup g$, $\max g$, etc.\ for the supremum, maximum, etc.\ of a real-valued function $g$ over its (effective) domain. For any real-valued increasing function $g$ on an interval $(a,b)$, where $-\infty \le a <  b \le +\infty$, denote by $\bar g$ its extension to $\R$ given by $\bar g(x) := g(b-)$ for $x \ge b$ and $\bar g (x) := g(a+)$ for $x \le a$. For a real $x$, put $x_+:=\max(x,0)$ and $x_-:=(-x)_+$, and use the same notation for functions. By convention, put $\frac{1}{0}:=+\infty$ and $\frac{C}{0}:=\R^d$ for any set $C \subset \R^d$ satisfying $0 \in \intr C $. Recall that $K(u)=\log \mathcal L$ is the cumulant generating function of $X_1$, finite on its effective domain $\mathcal D_{\mathcal L}$. 
%Denote by $g^*$ the Fenchel--Legendre transform (cf.~\eqref{eq: Legendre-%Fenchel}) of a function $g:\R^d \to (-\infty, +\infty]$. 

\begin{theorem} \label{thm: weighted}
Assume that $X_1$ is a random vector in $\R^d$ such that $0 \in \intr \mathcal D_{\mathcal L}$ and $f:[0,1]\to \R$ is a non-zero Lipschitz function. Then the sequence of random vectors $\big(\frac1n \sum_{k=1}^n f(\frac{k}{n}) X_k\big)_{n \ge 1}$ satisfies the LDP in $\R^d$ with the tight rate function $I_f$ that is the Legendre--Fenchel transform of the convex function 
%with the convex rate function $I_f:= E_f^*$, where $E_f$ is a convex function %given by 
\begin{equation} \label{E_f definition}
E_f(\lambda):=\int_0^1 K(\lambda f(t))dt, \quad \lambda \in \R^d.
\end{equation}

Moreover, if  $\dim(\supp (X_1))=d$,  the rate function satisfies
\begin{equation} \label{eq: weight rate f explicit}
I_f(x)=
\int_0^1 I(\nabla K( (\nabla E_f)^{-1}(x) f(t))) dt , \qquad x \in \nabla E_f(\intr D_f),
%\quad \text{where } D_f:=  \frac{\mathcal D_{\mathcal L} }{\max f_+}%\bigcap \frac{-\mathcal D_{\mathcal L} }{\max f_- }  ;
\end{equation}
where $D_f:=  \frac{\mathcal D_{\mathcal L} }{\max f_+} \cap \frac{-\mathcal D_{\mathcal L} }{\max f_- } $ and $\nabla E_f$ is an injective function (on its domain $\intr D_f$). For $d=1$, equality~\eqref{eq: weight rate f explicit} extends to 
\begin{equation} \label{eq: weight rate f 1-D}
I_f(x)= M_+{(x-\sup E'_f)}_+ + M_- {(x-\inf E'_f)}_-  + \int_0^1 I\big(K' \big(\overline{(E_f')^{-1}}(x) f(t)\big)\big) dt , \quad x \in \R,
\end{equation}
where $M_\pm:= \min \big (\frac{I_\infty(1)}{\max f_\pm}, \frac{I_\infty(-1)}{\max f_\mp} \big)$ and $K'(\pm I_\infty(\pm 1)):=K'(\pm I_\infty(\pm 1)\mp)$, with the symbol $\mp$ standing for the left/right limit. 
%$M_+$ and $M_-$ are given by $M_\pm:= \min \big (\frac{I_\infty(1)}{\max f_\pm}, \frac{I_\infty(-1)}{\max f_\mp} \big)$ and the values $K'(\pm I_\infty(\pm 1))$ are defined by continuity. 
\end{theorem}

\begin{remark} \label{rem: weighted}
Let us make some comments. 

\begin{enumerate}
\item \label{item: dim < d} Equality~\eqref{eq: weight rate f explicit} remains valid when $\dim(\supp (X_1))<d$, in which case the integrand on the r.h.s.\ still is a well-defined function, calculated by taking any element of the set $ (\nabla E_f)^{-1}(x) $. We will prove this together with the main case where  $\supp (X_1)$ has full dimension.

\item \label{item: simplification} If $\mathcal D_{\mathcal L} = \R^d$ and $\dim(\supp(X_1))=d$, then $\intr(\mathcal D_{I_f})= \nabla E_f(\intr D_f)$, hence in this case equality \eqref{eq: weight rate f explicit}  completely defines $I_f$ by lower semi-continuity; and if we additionally assume that $\conv( \supp(X_1)) = \R^d$, then $\nabla E_f(\intr D_f) = \R^d$ and if $d=1$, the first two terms in \eqref{eq: weight rate f 1-D} vanish. We will prove the claims of Items~\ref{item: simplification}) and~\ref{item: minimizer})  below  after proving Theorem~\ref{thm: weighted}.

\item \label{item: minimizer} Assume that $\dim(\supp(X_1))=d$.  Then for any $x \in \nabla E_f(\intr D_f)$, the function $h^{(x)}(t):=\int_0^t\nabla K( (\nabla E_f)^{-1}(x) f(s))ds$ is the unique minimizer of $I_D$ on $\{h \in BV[0,1]: \int_0^1 f dh = x\}$. If we additionally assume that $d=1$, $M_+<\infty$, $f \ge 0$ for simplicity, and $\argmax_{t \in [0,1]} f(t)$ has Lebesgue measure zero to fully distinguish from the usual case $f \equiv 1$ of equal weights, then for (say) any $x \ge \sup E_f'$, all the minimizers of $I_D$ are of the form
\[
h(t)= \frac{x - \sup E_f'}{\max f_+}g(t) + \int_0^t K'( M_+ f(s))ds, \qquad t \in [0,1],
\]
where $g \in BV[0,1]$ is a non-decreasing function such that $g(0)=0$, $g(1)=1$, and $dg$ is supported on $\argmax_t f(t)$. For example, if $\argmax_t f(t)$ has a unique element $t_0$, the only possible $g$ is $\I_{[t_0,1]}$. The singular part of such $h$ provides the first term on the r.h.s.\ of~\eqref{eq: weight rate f 1-D}.

Also note that in the case when $\mathcal D_{\mathcal L} = \R^d$, the function $h^{(x)}$ is the unique solution to the {\it Euler--Lagrange equation} for the Lagrangian $L(t,p):=I(p) +f(t) \lambda \cdot p$, where $t \in [0,1]$ and $p \in \R^d$, but the corresponding results of classical variational calculus (see, e.g., Cesari~\cite[Sections~2.2 and 2.7]{Cesari}) require the additional assumptions $f \in C^1[0,1]$ and $\mathcal D_I=\R^d$; the latter one is equivalent to $\conv(\supp (X_1))=\R^d$.

\item Let us give a concrete example. Assume that $X_1$ is a non-generate Gaussian vector in $\R^d$. We have $K(v)=\frac12 v^\top \Sigma v + \mu \cdot v$ and $I(v)=\frac12 (v-\mu)^\top \Sigma^{-1} (v-\mu)$ for $v \in \R^d$, where $\mu:= \E X_1$ and $\Sigma:=\E(X_1 X_1^\top) - \mu \mu^\top$. Then $\nabla K(v) = \Sigma v + \mu$ and $\nabla E_f(\lambda)= m_1 \mu + m_2 \Sigma \lambda$ for $\lambda \in \R^d$, where $m_i:= \int_0^1 f^i(s) ds$ for $i \in \{1,2\}$, hence $h^{(x)}(t) = t \mu  + m_2^{-1} (x-m_1 \mu) F(t)$ for $x \in \R^d$, where $F(t):=\int_0^t f(s)ds$. This gives $I_f(x)=\frac12 m_2^{-1} (x - m_1 \mu)^\top \Sigma^{-1} (x - m_1 \mu) $.
\end{enumerate}
\end{remark}

\begin{proof}[{\bf Proof of Theorem~\ref{thm: weighted}}]
We use the representation 
$$
\frac1n \sum_{k=1}^n f\Big(\frac{k}{n} \Big) X_k= \int_0^1 f d \Big(\frac1n S_{[n \cdot]} \Big)$$ 
for the kernel-weighted sums. Any Lipschitz function is absolutely continuous, and the integration by parts formula (\cite[Theorem~3.36]{Folland}) yields $\int_0^1 f d h = f(1) h(1) - \int_0^1 f' h dt$ for $h \in BV[0,1]$. Hence the functional $h \mapsto \int_0^1 f d h $ on $BV[0,1]$ is continuous in the metric $\rho_*$ since $f'$ is bounded and $L^\infty[0,1]$ is dual to $L^1[0,1]$. Therefore, by the usual contraction principle (\cite[Theorem~4.2.1]{DemboZeitouni}) it follows from Proposition~\ref{prop: LDP weak} that the sequence of kernel-weighted sums satisfies the LDP  in $\R^d$ with the tight rate function  
\begin{equation} \label{eq: weight rate f 2}
I_f(x)=\inf \limits_{h \in BV([0,1]; \, \R^d): \, \int_0^1 f d h = x} I_D(h), \qquad x \in \R^d,
\end{equation} 
where the infimum is always attained at some $h$. Let us {\it compute} this function. 

Denote by $\chi_C$ the {\it convex-analytic} characteristic function of a set $C \subset \R^d$, defined to be $0$ on the set and $+\infty$ on its complement. By definition~\eqref{eq: Legendre-Fenchel} of the Legendre--Fenchel transform, which we denote by $^*$, for any $\lambda \in \R^d$,
\begin{align*}
I_f^*(\lambda) &= \sup_{x \in \R^d} \big( \lambda \cdot x - I_f(x) \big) = \sup_{x \in \R^d} \bigg( \lambda \cdot x - \inf_{h \in BV} \bigg[ I_D(h) + \chi_{\{x\}} \bigg( \int_0^1 f dh \bigg) \bigg] \bigg) \\
&=\sup_{x \in \R^d} \sup_{h \in BV} \bigg( \lambda \cdot x  - I_D(h) - \chi_{\{x\}} \bigg( \int_0^1 f dh \bigg) \bigg) \\
&=\sup_{h \in BV} \bigg(  \int_0^1 \lambda f \cdot dh - I_D(h) \bigg), 
\end{align*}
where the last equality follows after interchanging the suprema. Thus, $I_f^*(\lambda)=I_D^*(\lambda f)$, where $I_D^*$ is the Legendre--Fenchel transform of the function $I_D$ on $BV$, defined by the standard duality (given by the respective integral) between the spaces  $BV[0,1]$  and $C[0,1]$. 

By representation \eqref{eq: I_D = integral}, we have
\[
I_f^*(\lambda) = \sup_{h \in AC_0} \bigg( \int_0^1 \lambda f \cdot h' dt - \int_0^1 I(h')dt \bigg) + \sup_{h \in BV: \, h_a=0}  \bigg( \int_0^1 \lambda f \cdot d h - \int_{\S^{d-1}} I_\infty(\ell) \sigma^{h}(d \ell)\bigg).
\]
To find the first supremum, we use Proposition IX.2.1 in the book by Ekeland and T\'emam~\cite{EkelandTemam}, which computes the Legendre--Fenchel transform of the functional $g \mapsto \int_0^1 I(g(t))dt$ on $L^1([0,1]; \R^d)$, defined by the standard duality (given by the respective integral) between the spaces $L^1[0,1]$  and $L^\infty[0,1]$. This result applies, in the notation and terminology of~\cite{EkelandTemam}, with $\alpha =1$, $u_0\equiv\E X_1$, and the integrand $f=I$, which is non-negative and normal (i.e., lower semi-continuous) on $B=\R^d$. This gives, by $I^*=K$,
\begin{equation} \label{eq: I_f^* computation}
I_f^*(\lambda) = \int_0^1 K(\lambda f(t))dt  + \sup_{h \in BV: \, h_a=0}  \bigg( \int_0^1 \lambda f \cdot d h - \int_{\S^{d-1}} I_\infty(\ell) \sigma^{h}(d \ell)\bigg),
\end{equation}
where the first term is $E_f(\lambda)$.

Consider the second term in \eqref{eq: I_f^* computation}. For $h \in BV[0,1]$, put $h^\pm(t):=d h([0,t]\cap \{\pm f > 0\})$ for $t \in [0,1]$ and $h^=:=h-h^+-h^-$. We have $h^\pm, h^=\in BV[0,1]$. It follows from \eqref{eq: TV measure} that $V^h=V^{h^+} + V^{h^-} +V^{h^=}$, hence $\sigma^h=\sigma^{h^+} + \sigma^{h^-} +\sigma^{h^=}$ by the definition of directional total variation. Therefore,
\begin{align} \label{eq: f dh split}
\int_0^1 \lambda f \cdot d h - \int_{\S^{d-1}} I_\infty(\ell) \sigma^{h}(d \ell) &= \sum_{\varsigma \in \{+,-,=\}} \bigg[ \int_0^1 \lambda f \cdot \dot { h^\varsigma} \, d V^{h^\varsigma} - \int_{\S^{d-1}} I_\infty(\ell) \sigma^{h^\varsigma}(d \ell)\bigg] \notag \\
&\le \sum_{\varsigma \in \{+,-\}} \bigg[  \int_0^1  {(\varsigma \max f_\varsigma \cdot  \lambda \cdot \dot { h^\varsigma})}_+ \, d V^{h^\varsigma} - \int_{\S^{d-1}} I_\infty(\ell) \sigma^{h^\varsigma}(d \ell) \bigg]  \notag \\
&= \sum_{\varsigma \in \{+,-\}} \bigg[ \int_{\S^{d-1}} \big ( {(\varsigma \max f_\varsigma  \cdot \lambda \cdot \ell)}_+ - I_\infty(\ell)\big )\sigma^{h^\varsigma}(d \ell) \bigg], 
\end{align}
where the inequality follows from $0 \le \pm f \le \max f_\pm$ $(d V^{h^\pm})$-a.e. We estimate the integrands using that  $I_\infty(\ell)=\sup_{u \in \mathcal{D}_{\mathcal L} }   u \cdot \ell = \sup_{u \in \cl \mathcal{D}_{\mathcal L} }   u \cdot \ell$, and recall that $D_f =  \frac{\mathcal D_{\mathcal L} }{\max f_+} \cap \frac{-\mathcal D_{\mathcal L} }{\max f_- } $. Then
\begin{equation} \label{eq: singular term}
\int_0^1 \lambda f \cdot d h - \int_{\S^{d-1}} I_\infty(\ell) \sigma^{h}(d \ell) \le \chi_{\cl \mathcal D_{\mathcal L} } \big(\max f_+  \lambda \big) + \chi_{\cl \mathcal D_{\mathcal L} } \big(-\max f_-  \lambda \big) = \chi_{\cl D_f}(\lambda).
\end{equation}

By \eqref{eq: I_f^* computation} and \eqref{eq: singular term}, we have $ I_f^*(\lambda) = E_f(\lambda)$ for $\lambda \in \cl D_f $. On the other hand, $  I_f^*(\lambda) \ge E_f(\lambda) =+\infty$ for $\lambda \not \in \cl D_f $ because for such $\lambda$, $K(\lambda f(t) ) = +\infty$ for $t$ in a non-empty interval since $f$ is continuous on $[0,1]$ and $\mathcal D_K=\mathcal D_{\mathcal L}$. All together, we get $ I_f^*  = E_f $.

The function $E_f$ is convex on $\R^d$ as a mixture of convex functions $\lambda \mapsto K(\lambda f(t))$. Since $K$ is lower semi-continuous by Fatou's lemma, so is $E_f$, again by Fatou's lemma. Therefore, $I_f^*= E_f$ yields the required identity $I_f=E_f^*$ by \cite[Theorem 12.2 and Corollary 12.1.1]{Rockafellar}. 

Next we prove formula~\eqref{eq: weight rate f explicit} for $I_f$. By Theorem~26.4 in~\cite{Rockafellar}, which applies because $E_f$ is a continuous convex function differentiable on $\intr D_f$ (since so is $K$ on $\intr \mathcal D_{\mathcal L}$), we have
$$
I_f(x) = x \cdot (\nabla E_f)^{-1}(x) - E_f((\nabla E_f)^{-1}(x)), \quad x \in \nabla E_f(\intr D_f).
$$
The result used also states that the r.h.s.\ is well-defined even if $(\nabla E_f)^{-1}(x) $ contains more than one element, in which case we shall understand the r.h.s.\ replacing $(\nabla E_f)^{-1}(x) $ by any $\lambda \in (\nabla E_f)^{-1}(x) $, and the resulting value does not depend on the particular choice of~$\lambda$. This justifies Remark~\ref{rem: weighted}.\ref{item: dim < d}.

For any $x \in \nabla E_f(\intr D_f)$ and $\lambda \in (\nabla E_f)^{-1}(x) $, we have 
\[
x   =\nabla E_f(\lambda)  = \nabla \bigg( \int_0^1 K(\lambda f(t) ) dt \bigg)  =  \int_0^1  \nabla K(\lambda f(t))  f(t)  dt,
\]
because the cumulant generating function $K$ is smooth on $\intr \mathcal D_{\mathcal L}$ and $\lambda f(t) \in \intr \mathcal D_{\mathcal L}$ for every $t \in [0,1]$ by $\lambda \in (\nabla E_f)^{-1}(x) \subset \intr D_f$. Then 
\begin{align*}
I_f(x) &=x \cdot  \lambda - \int_0^1 K(\lambda f(t) )  dt  \\
&= \int_0^1 \big( \nabla K(\lambda f(t)) \cdot \lambda f(t) -  K(\lambda f(t) ) \big) dt = \int_0^1 I(\nabla K(\lambda f(t))) dt,
\end{align*}
where in the last equality  we  applied~\cite[Theorem~26.4]{Rockafellar} again. This proves~\eqref{eq: weight rate f explicit}.

Furthermore, $\dim(\supp (X_1))<d$ if and only if $X_1$ is supported on a hyperplane, in which case $K$, and hence $E_f$, is constant along the lines orthogonal to the hyperplane. Therefore $\nabla E_f$ cannot be injective in this case. On the contrary, if  $\dim(\supp (X_1))=d$, then it follows by a standard application of the Cauchy--Schwartz inequality  that the Hessian  of  $K$ is positive-definite on the interior of its effective domain. Then the same holds for $E_f$, which is a mixture of  functions $\lambda \mapsto K(\lambda f(t))$ which are positively definite  when $f(t)\neq 0$. This implies that $\nabla E_f$ is injective, as claimed, since the scalar function $t \mapsto \nabla E_f ((1-t) \lambda_1 + t \lambda_2) \cdot (\lambda_2 - \lambda_1) $ for $t \in [0,1]$ has strictly positive derivative  whenever $\lambda_1, \lambda_2 \in \R^d$ are distinct. 
%we have $G(\lambda_2)-G(\lambda_1)= \int_0^1 \nabla G((1-t) \lambda_1 + %t \lambda 2) (\lambda_2 - \lambda_1) dt$.

It remains to prove equality \eqref{eq: weight rate f 1-D}. Here $d=1$ and $\intr D_f=(-M_-, M_+)$. For $x \in E_f'(\intr D_f)$, \eqref{eq: weight rate f 1-D} reduces to equality~\eqref{eq: weight rate f explicit} and there is nothing to prove. Since $K$ is continuously differentiable and convex on $\intr \mathcal D_{\mathcal L}$, so is $E_f$ on $\intr D_f$. Therefore $E_f'(\intr D_f)=(\inf E_f', \sup E_f')$. Assume that the complement of this set is non-empty and consider an $x$ from there. W.l.o.g., we can assume that $\sup E_f' <\infty$ and prove \eqref{eq: weight rate f 1-D} only on $[\sup E_f', \infty)$. Also, assume that $X_1$ is not constant, otherwise the claim is trivial.

We can check that equality \eqref{eq: weight rate f 1-D} holds true for $x =\sup E_f'$  by taking $x \nearrow \sup E_f' $ in \eqref{eq: weight rate f 1-D} and using lower semi-continuity of $I$ combined with  the facts that $I_f$ increases on $[m_1 \mu, +\infty) $ and $\sup E_f'> E_f'(0)=m_1 \mu$, where $m_1=\int_0^1 f(t)dt$ and $\mu = \E X_1$. Therefore it remains to prove \eqref{eq: weight rate f 1-D} for $x >\sup E_f'$. We have $I_f(x)= E_f^*(x)=\sup_{\lambda \in \R^d} (x \cdot \lambda - E_f(\lambda))$. If $M_+ = \infty$, this gives $I_f(x)=\infty$ for  $x >\sup E_f'$, matching~\eqref{eq: weight rate f 1-D}. If $M_+<\infty$, this gives, by taking into account that $E_f(\lambda) =+\infty$ for $\lambda > M_+$, that $I_f(x)=M_+x   - E_f(M_+)$ for $x >\sup E_f'$. By the equality $I(K'(u))= u K'(u)-K(u)$ for $u \in \cl (-M_-, M_+)$, we get
\begin{align*}
I_f(x)&=M_+ x - M_+ \int_0^1 f(t) K'(M_+ f(t)) dt + \int_0^1 \big ( M_+ f(t) \cdot K'(M_+ f(t)) - K(M_+ f(t)) \big) dt\\
&= M_+(x - E_f'(M_+-)) + \int_0^1 I(K'(M_+ f(t))dt,
\end{align*}
which coincides with  \eqref{eq: weight rate f 1-D}  for $x >\sup E_f'$ since $\sup E_f'= E_f'(M_+-)$ and $\overline{(E_f')^{-1}}(x) = M_+$.
\end{proof}

\begin{proof}[{\bf Proof of Remark~\ref{rem: weighted}}]
\ref{item: simplification}) If  $\dim(\supp(X_1))=d$, then a standard application of H\"older's inequality implies strict convexity of $K$ on $\mathcal D_{\mathcal L}$.
So is $E_f$ on its effective domain $\intr D_f$. If we additionally require that $\mathcal D_{\mathcal L} = \R^d$, then  $D_f = \R^d$ and hence $\intr(\mathcal D_{I_f})= \nabla E_f(\intr D_f)$ by Theorem~26.5 in \cite{Rockafellar}. If we further assume that  $\conv(\supp(X_1))=\R^d$, then $\conv(\supp(u \cdot X_1))=\R$ for any non-zero $u \in \R^d$. This readily implies $\lim_{t \to \infty} K(tu)/t =+\infty$. Therefore Theorem~26.6 from \cite{Rockafellar} applies, ensuring that $\nabla E_f$ is a homeomorphism from $\R^d$ to $\R^d$.

\ref{item: minimizer}) Fix any $x \in \R^d$ and $h \in BV[0,1]$ such that  $x=\int_0^1 f dh$. Put $x_a:= \int_0^1 f d h_a$ and $x_s:=\int_0^1 f dh_s $. For any $\lambda \in \R^d$, by Fenchel's inequality we have 
\begin{equation} \label{eq: Fenchel}
I(h'(t)) + K(\lambda f(t) ) \ge \lambda f(t) \cdot h'(t), \quad \text{a.e. } t \in [0,1],
\end{equation} 
hence $I_D(h_a) \ge \lambda x_a - E_f(\lambda)$. Combining this with inequality \eqref{eq: singular term} and optimizing over $\lambda \in \cl D_f$ yields
\begin{equation} \label{eq: I_D(h) = I_f*}
I_D(h) \ge \sup_{\lambda \in \cl  D_f} \big( \lambda x - E_f(\lambda) \big) = I_f(x),
\end{equation}
where the equality follows from  the facts that $I_f=E_f^*$ and  $\mathcal D_{E_f} \subset  \cl  D_f$. 

If $x \in \nabla E_f(\intr D_f)$, the supremum in \eqref{eq: I_D(h) = I_f*} is attained at $\lambda = (\nabla E_f)^{-1}(x) $; recall that $\nabla E_f$ is injective due to the assumption $\dim(\supp (X_1))=d$. Therefore, $I_D(h)=I_f(x)$ implies that \eqref{eq: singular term}  is an equality and \eqref{eq: Fenchel} is an a.e.\ equality for this particular $\lambda$. Since $\lambda \in \intr D_f$, we have $\lambda f(t) \in \intr \mathcal D_K$ for every $t \in [0,1]$, hence Fenchel's inequality~\eqref{eq: Fenchel} is an equality if and only if $h'(t)=\nabla K(\lambda f(t))$ for a.e.\ $t$ (\cite[Theorem~23.5]{Rockafellar}). Also, since  $\max f_\pm \cdot  \lambda \in \intr \mathcal D_{\mathcal L}$, inequality \eqref{eq: singular term} is strict when $h_s \neq 0$ because the integrands in the last line of~\eqref{eq: f dh split} are strictly negative. All together, this means that  $h(t)=\int_0^t \nabla K ((\nabla E_f)^{-1}(x) f(s)) ds$ is the unique minimizer of $I_D$, i.e.\ $h=h^{(x)}$, as claimed.

If $d=1$, $\sup E_f'< \infty$, and $M_+<\infty$, the implication $(I_D(h)=I_f(x)) \Rightarrow (h=h^{(x)})$ extends to $x = \sup E_f'$ (corresponding to $\lambda = M_+$) by continuity as in the corresponding argument in the proof of Theorem~\ref{thm: weighted}. For $\lambda = M_+$ and when $\argmax_t f(t)$  is a singular set,  Fenchel's inequality~\eqref{eq: Fenchel} still is an equality if and only if $h'(t)=\nabla K(M_+ f(t))$ for a.e.\ $t$ because $M_+ f(t) \in \intr \mathcal D_K$ for a.e.\ $t$. 
However, for $\lambda = M_+$ and when $f \ge 0$, the last line of \eqref{eq: f dh split} is zero, and thus \eqref{eq: f dh split} is an equality if and only if $dh_s$ is a non-negative finite measure supported on $\argmax_t f(t)$. Thus, for $x > \sup E_f'$, the equality $I_D(h)=I_f(x)$ implies that $h_a=h^{(\sup E_f')}$ and $h_s(1)=\frac{x - \sup E_f'}{\max f_+}$, as claimed.
\end{proof}

\section*{Acknowledgements} I am grateful to Anatoly Mogulskii for extended explanations of his works on large deviations and for his comments on the current paper. I also thank 
%Adam Jakubowski for comments on the $S$-topology, 
Mikhail Lifshits for discussions on large deviations, G\"unter Last for explanations on his local Steiner-type formula, and Chang-Han Rhee for the comments on~\cite{Bazhba+}. I am indebted to the anonymous referee for very useful suggestions to improve presentation of the paper.

\bibliographystyle{plain}
\bibliography{ldp}

\end{document}